\documentclass[a4paper,12pt]{article}
\usepackage{amsthm}
\usepackage{amsmath,amsfonts,amssymb}
\usepackage{a4wide}
\usepackage[usenames,dvipsnames]{color}
\usepackage[all,cmtip]{xy}
\usepackage[verbose,colorlinks=true,linktocpage=true,linkcolor=blue,citecolor=blue]{hyperref}
\usepackage{footnote}
\usepackage{tikz}
\usepackage{extarrows}
\usepackage[title]{appendix}
\usepackage{marvosym}

\usepackage{indentfirst} 
\usepackage{geometry}
\newtheorem{theo}{{Theorem}}[section]
\newtheorem{lemm}[theo]{Lemma}
\newtheorem{rema}[theo]{Remark}
\newtheorem{defi}[theo]{Definition}
\newtheorem{coro}[theo]{Corollary}
\newtheorem{prop}[theo]{Proposition}

\numberwithin{equation}{section}

\hypersetup{
    colorlinks=true, 
    linkcolor=blue, 
    urlcolor=blue, 
}

\allowdisplaybreaks[1]

\newcommand{\hongda}[1]{\textcolor{blue}{#1}}

\begin{document}

\title{ Isomorphism between the R-matrix and Drinfeld presentations of quantum affine superalgebra for type $\boldsymbol{A}$}

\author{ Pengfa Xu${}^1$, Hongda Lin${}^2$ and Honglian Zhang${}^{1,}$\thanks{Corresponding Author.~~Email:~hlzhangmath@shu.edu.cn}}
\maketitle

\begin{center}
\footnotesize
\begin{itemize}
\item[1] Department of Mathematics, Shanghai University, Shanghai 200444, PR~China.
\item[2] Shenzhen International Center for Mathematics, Southern University of Science and Technology, Shenzhen 518055, PR~China. 

\end{itemize}
\end{center}

\begin{abstract}
 In this paper, we establish an explicit isomorphism between Drinfeld and R-matrix presentations of the quantum affine superalgebra associated with the general linear Lie superalgebra $\mathfrak{gl}(m|n)$. This result can be viewed as a supersymmetric analogue of the isomorphism between two presentations of the quantum affine algebra $U_{q}(\widehat{\mathfrak{gl}(n)})$ proposed by Frenkel-Ding. Furthermore, employing the quantum Berezinian of $U_{q}(\widehat{\mathfrak{gl}(m|n))}$, we construct the R-matrix presentation of the quantum affine superalgebra $U_{q}(\widehat{\mathfrak{sl}(m|n))}$ and prove the isomorphism between its R-matrix and Drinfeld presentations. Our proof based on the commutative relation between Gaussian generators, $\mathbb{A}$-form and quantum Berezinian. \

\noindent{\textbf{Keywords:}}
Quantum affine superalgebra; R-matrix presentation; Drinfeld presentation; $\mathbb{A}$-form; quantum Berezinian.
\end{abstract}

\section{Introduction}
The \textit{quantum affine algebra} $U_q(\widehat{\mathfrak{g}}) $ is a quantum deformation of the universal enveloping algebra associated with an affine Kac-Moody algebra $\widehat{\mathfrak{g}}$. It is well-known that $U_q(\widehat{\mathfrak{g}})$ admits three different presentations. The original definition was independently introduced by Drinfeld \cite{Dr85} and Jimbo \cite{Ji85} in terms of Chevalley generators and $q$-Serre relations, a formulation now refered to as the \textit{Drinfeld-Jimbo presentation}. Drinfeld \cite{Dr87} also gave another presentation of the quantum affine algebra in terms of current generators in 1987, which is commonly known as \textit{Drinfeld (new) presentation}. Subsequently, the \textit{$R$-matrix presentation} of the quantum affine algebra was proposed by Reshetikhin and Semenov-Tian-Shansky \cite{RS90} and further developed by I. Frenkel and Reshetikhin \cite{FR92}. The defining relations of this presentation are written in the form of $RLL$ relations associated with a trigometrix $R$-matrix.

Considerable advances has been made in investigating isomorphisms among different presentations of quantum affine algebras. Beck's pioneering work \cite{Be94} proved the isomorphism between the Drinfeld–Jimbo and Drinfeld presentations for untwisted algebras. This work was subsequently extended to the twisted case by Jing and Zhang \cite{JZ07,JZ16}. In a parallel development, Damiani \cite{Da12,Da15} provided an alternative construction of this isomorphism using methods distinct from Beck's. In 1993, Frenkel and Ding \cite{DF93} constructed the isomorphism between the Drinfeld and $R$-matrix presentations in the case of simple Lie algebra $\mathfrak{g}$ of type $\boldsymbol{A}$. Following the approach of Ding and Frenkel, Jing, Liu and Molev \cite{JLM20-1,JLM20-2} generalized this result to the Lie algebras $\mathfrak{g}$ of types $\boldsymbol{B}$, $\boldsymbol{C}$ and $\boldsymbol{D}$, with a similar result for Yangians \cite{JLM18}. These developments have established the isomorphism among these presentations,  providing multiple perspectives for studying the representation theory of quantum affine algebras and enriching the framework for understanding these algebraic structures.

To incorporate a natural $\mathbb{Z}_2$-grading, quantum affine superalgebras are introduced as an extension of quantum affine algebras through the inclusion of additional odd generators. In \cite{Ya99}, Yamane introduced the Drinfeld-Jimbo presentations of quantum affine superalgebras together with the complete Serre relations by deforming the universal enveloping algebras of affine Lie superalgebras of types $\boldsymbol{A}$--$\boldsymbol{G}$. Adopting Beck’s approach \cite{Be94}, the Drinfeld presentation of quantum affine superalgebras were derived for type $\boldsymbol{A}$ \cite{Ya99}, types $\boldsymbol{BCD}$ \cite{BFK25} and exceptional type $\boldsymbol{D}(2,1;x)$ \cite{HSTY08}. However, these studies only provided the surjective homomorphism from Drinfeld to Drinfeld-Jimbo presentations except for $\boldsymbol{D}(2,1;x)$. Until recently, Lin, Yamane and Zhang in \cite{LYZ24} demonstrated this isomorphism for type $\boldsymbol{A}$ whose Dynkin diagram does not contain two neighboring odd roots, while another paper \cite{WLZ25} presented for the case of type $\boldsymbol{B}$ at standard parities.

It is noteworthy that Zhang \cite{Zyz97}, and Fan, Hou and Shi \cite{FHS97} constructed the Drinfeld presentation of quantum affine superalgebras $U_q \bigl(\widehat{\mathfrak{gl}(m|n)}\bigr)$. However, these constructions are predicated on the assumption that the Frenkel-Ding isomorphism theorem holds in the super case and do not explicitly present the complete Serre relations. Furthermore, Zhang \cite{Zhf16} applied the $R$-matrix presentation of the quantum loop superalgebra associated with the Lie superalgebra $\mathfrak{gl}(m|n)$ to investigate its finite-dimensional representations and their tensor products. Zhang also pointed out that the $R$-matrix presentation is isomorphic to the Drinfeld presentation in the loop case, although the proof of injectivity was not provided. In 2021,  Tsymbaliuk \cite{Ts21} employed the shuffle algebra approach to provide a positive part of PBWD basis for type $\boldsymbol{A}$ quantum affine superalgebra. More recently, Hu, Jing and Zhong in \cite{HJZ24} provided the Drinfeld presentation of two-parameter quantum affine superalgebra $U_{p,q}(\widehat{\mathfrak{gl}(m|n)})$. In \cite{WLZ24}, Wu, Lin and Zhang established the isomorphism between the Drinfeld and $R$-matrix presentations of quantum affine superalgebra for type $\mathfrak{osp}(2m+1|2n)$. 

This paper aims to extend the Frenkel-Ding isomorphism theorem to the quantum affine superalgebras of type $\boldsymbol{A}$. We first prove the isomorphism between the $R$-matrix and Drinfeld presentations of the quantum affine superalgebra $U_{q}(\widehat{\mathfrak{gl}(m|n))}$, based on the Guass decomposition of the generator matrices in the $R$-matrix presentation. The argument proceeds in two main stages: first, we verify that the generators obtained via Gauss decomposition satisfy all relations in the Drinfeld presentation; second, we prove the injectivity of the induced homomorphism. Injectivity is shown first for the $\mathbb{A}$-forms of both two presentations. Then, since the homomorphism preserves the $\mathbb{A}$-form, its injectivity follows. Building on this, we use the quantum Berezinian of $U_{q}(\widehat{\mathfrak{gl}(m|n))}$ as defined by Jing, Li and Zhang \cite{JLZ25} to introduce the $R$-matrix presentation of $U_{q}(\widehat{\mathfrak{sl}(m|n))}$, and prove its isomorphism to the Drinfeld presentation. This completes the isomorphism between the $R$-matrix and the Drinfeld presentation of the quantum affine superalgebras of type $\boldsymbol{A}$. Our isomorphism, combined with the isomorphism between the Drinfeld–Jimbo and Drinfeld presentations established in \cite{LYZ24}, establishes the complete equivalence among the three presentations of type $\boldsymbol{A}$ quantum affine superalgebras at standard 01-sequences.

This paper is organized as follows. In Section 2, we recall the Drinfeld and R-matrix presentations of quantum affine superalgebra $U_q(\widehat{\mathfrak{gl}(m|n)})$ and formulate our main theorem. We subsequently prove that the map established in our theorem forms a isomorphism between these two presentations in Section 3. Finally, in Section 4, we introduce the $R$-matrix presentation of $U_{q}(\widehat{\mathfrak{sl}(m|n))}$, and establish its isomorphism to the Drinfeld presentation.

Throughout this paper, we adopt the conventions that $\mathbb{C}$ denotes the set of complex numbers, $\mathbb{C}^*$ represents the set of nonzero complex numbers, $\mathbb{N}$ corresponds to the set of non-negative integers, $\mathbb{Z}$ denotes the set of integers, $\mathbb{Z}_+$ signifies the set of positive integers, and $\mathbb{Z}^*$ represents the set of nonzero integers. We write $\mathbb{Z}_2=\mathbb{Z}/2\mathbb{Z}=\{\bar{0},\bar{1}\}$.

\section{Quantum affine superalgebra}

In this section, we summarize some notations and known facts that will be used throughout the paper and review the Drinfeld and $R$-matrix presentations of the quantum affine superalgebra associated with the general linear Lie algebra $\mathfrak{gl}(m|n)$. More details can be found in \cite{BCFK22,DF93,JLZ25,Ka77,LWZ24,LZ25,Mu12,VdL89,Ya99,Zhf14,Zyz97}. 

\subsection{Notations and Setup}

A \textit{superspace} is a $\mathbb{Z}_2$-graded vector space $V=V_{\bar{0}} \oplus V_{\bar{1}}$ over $\mathbb{C}$. Elements of $V_{\bar{0}}$ are referred to as even, and those of $\emph{V}_{\bar{1}}$ as odd. An element $x\in V$ is said to be \textit{homogeneous} if $x$ is either even or odd. The parity $|\cdot|$ of a homogeneous element $x\in V$ is defined by
\begin{equation*}
    |x|=\begin{cases}
        0, &\text{if}~~x\in V_{\bar{0}}; \\
        1, &\text{if}~~x\in V_{\bar{1}}.
    \end{cases}
\end{equation*}

An (associative) \textit{superalgebra} is a $\mathbb{Z}_2$-graded algebra $\mathfrak{A}=\mathfrak{A}_{\bar{0}}\oplus\mathfrak{A}_{\bar{1}}$ over $\mathbb{C}$. Then  $\mathfrak{A}$ is a superspace whose multiplication satisfies $xy\in \mathfrak{A}_{a+b}$ if $x\in \mathfrak{A}_a$, $y\in \mathfrak{A}_b$ for $a,b\in\mathbb{Z}_2$. Let $\mathfrak{A}_1,\mathfrak{A}_2$ be two superalgebras over $\mathbb{C}$. The tensor product $\mathfrak{A}_1\otimes \mathfrak{A}_2$ also forms a superalgebra subject to the multiplication rule defined by:
$$
(x_1 \otimes x_2)(y_1 \otimes y_2)=(-1)^{|x_2||y_1|}(x_1y_1 \otimes x_2y_2),
$$
for homogeneous elements $x_1,y_1\in\mathfrak{A}_1$, $x_2,y_2\in\mathfrak{A}_2$. 

A \textit{Lie superalgebra} is a superspace $\mathfrak{L}=\mathfrak{L}_{\bar{0}} \oplus \mathfrak{L}_{\bar{1}}$ over $\mathbb{C}$ equipped with a bilinear operation $[\,\cdot,\, \cdot\,]$, satisfying the following identities:
\begin{align*}
    &[x,\,y]=-(-1)^{|x||y|}[y,\,x],  &&(\text{skew-supersmmetric}) \\
    &(-1)^{|x||y|}[[x,\,y],\,z]+(-1)^{|y||z|}[[y,\,z],\,x]+(-1)^{|z||x|}[[z,\,x],\,y]=0,  &&(\text{super-Jacobi identity}) 
\end{align*}
where $x,y,z\in\mathfrak{L}$ are homogeneous elements. 

\vspace{1em}
Given the nonnegative integers $m,n$ such that $N=m+n\geqslant 2$. 
Define the set $I_{N}:=\{1,\cdots, N \}$, which we will simply denote by $I$. The parity of $i\in I$ is defined by
\begin{equation*}
    |i|=\begin{cases}
        0, &\text{if}~~1\leqslant i\leqslant m; \\
        1, &\text{if}~~m+1\leqslant i\leqslant N. 
    \end{cases}
\end{equation*}
Consider the superspace $\mathcal{V}:=\mathbb{C}^{N}$ with its standard basis $\{v_i\}_{i\in I}$, where $|v_i|=|i|$. The endomorphism ring $\operatorname{End}\mathcal{V}$ carries a natural structure of a superspace with the standard basis $E_{ij}$ of parity $|i|+|j|$ for $i,j \in I$.

Given a matrix $R=\sum_i r_{i}\otimes r^{i}\in \operatorname{End}\mathcal{V}^{\otimes 2}$ and an integer $t\geqslant 2$, we define the element $R_{ab}$ of $\operatorname{End}\mathcal{V}^{\otimes t}$ for $1\leqslant a<b\leqslant t$ by 
\begin{gather*}
    R_{ab}=\sum 1^{\otimes (a-1)}\otimes r_i\otimes 1^{\otimes (b-a-1)}\otimes r^i\otimes 1^{\otimes (t-b)}.
\end{gather*}
Let $P$ be the $\mathbb{Z}_2$-graded permutation operator on the tensor product $\mathcal{V}\otimes \mathcal{V}$ given by
$$
P=\sum_{i,j\in I}(-1)^{|j|}E_{ij}\otimes E_{ji}\in \operatorname{End}\mathcal{V}^{\otimes 2}.
$$
We then define $R_{ba}=P_{ab} R_{ab}P_{ab}$. 

Let $\mathfrak{A}$ be a superalgebra over $\mathbb{C}$. For a $\mathbb{Z}_2$-graded matrix 
$X=\sum_{i,j\in I} E_{ij}\otimes x_{ij}\in \operatorname{End}\mathcal{V}  \otimes\mathfrak{A},$
its supertranspose $X^{\rm st}$ is defined as
$$
X^{\rm st}=\sum_{i,j\in I} E_{ij}^{\rm st}\otimes x_{ij}\quad \text{with}\quad E_{ij}^{\rm st}=(-1)^{|i||j|+|j|}E_{ji}.
$$
For any $1\leqslant a\leqslant t$, we denote $X_a$ by  the matrix $X$ associated with the $a$-th copy of $\mathrm{End}\mathcal{V}$, which is given by 
$$X_a=\sum\limits_{i,j\in I}1^{\otimes (a-1)} \otimes E_{ij}\otimes 1^{\otimes (t-a)} \otimes x_{ij} \in \mathrm{End}\mathcal{V}^{\otimes t}\otimes \mathfrak{A}.$$ 

\vspace{0.5em}
The superspace $\operatorname{End}\mathcal{V}$, when endowed with the standard supercommutator: 
\begin{gather*}
    \left[E_{ij},\,E_{kl}\right]=\delta_{jk}E_{il}-\varepsilon_{ij;kl}E_{kj},\quad \text{for all}~~i,j,k,l\in I,
\end{gather*}
forms a Lie superalgebra structure, that is, \textit{general linear Lie superalgebra} $\mathfrak{gl}(m|n)$. Here, $\varepsilon_{ij;kl}=(-1)^{(|i|+|j|)(|k|+|l|)}$ for $i,j,k,l\in I$ and $\delta_{ij}$ denotes the Kronecker function.

The \textit{supertrace} map $\mathbf{str}$ is a $\mathbb{C}$-value operator acting on $\mathfrak{gl}(m|n)$, which is determined by extending $E_{ij}\mapsto (-1)^{|i|}\delta_{ij}$ for $i,j\in I$. Then the \textit{special linear Lie superalgebra} $\mathfrak{sl}(m|n)$ is defined as 
$$\mathfrak{sl}(m|n):=\{x\in\mathfrak{gl}(m|n)|\mathbf{str}(x)=0\}. $$
Let $\mathcal{I}=\sum_{i\in I}E_{ii}$ be the identity matrix. 
The Lie superalgebra $\mathfrak{sl}(m|n)$ is simple except that if $m=n$, it has a one-dimensional ideal $\langle \mathcal{I}\rangle$ spanned by scalar matrices $\lambda \mathcal{I}$, $\lambda \in\mathbb{C}$. 

Recall from \cite{Ka77} that the family $\boldsymbol{A}(m,n)$ of classical simple Lie superalgebras is given by:
\begin{align*}
\boldsymbol{A}(m-1,n-1)&=\mathfrak{sl}(m|n)\qquad~~\text{for}~~m\neq n,\quad m,n\geqslant 1;\\
\boldsymbol{A}(n-1,n-1)&=\mathfrak{sl}(n|n)/\langle \mathcal{I}\rangle := \mathfrak{psl}(n|n),\quad \text{for} ~~n>1.
\end{align*}
We say that a Lie superalgebra $\mathfrak{a}$ is of type $\boldsymbol{A}$ if $\mathfrak{a}=\mathfrak{gl}(m|n)$, $\mathfrak{sl}(m|n)$ or $\mathfrak{psl}(n|n)$. Let $\mathfrak{h}:=\text{Span}_{\mathbb{C}}\{h_i:=(-1)^{|i|}E_{ii}\ |\ i\in I\}$ be the \textit{Cartan sub-superalgebra} of $\mathfrak{gl}(m|n)$. Define the linear functions $\epsilon_i\in\mathfrak{h}^{\ast}$ ($i\in I$) given by $\epsilon_i(h_j)=(-1)^{|i|}\delta_{ij}$.
The parity of $\epsilon_i$ is $|\epsilon_i|=|i|$. 
Put $I'=I\setminus\{N\}$. For $i\in I'$, we denote 
$\alpha_i=\epsilon_i-\epsilon_{i+1}$  to be the \textit{simple roots} and $h_{\alpha_i}=h_i-h_{i+1}$ the \textit{simple coroots} of type $\boldsymbol{A}$. We call $\mathcal{Q}=\sum_{i\in I'}\mathbb{Z}\alpha_i$   the \textit{root lattice} and $\mathcal{Q}^{\vee}=\sum_{i\in I'}\mathbb{Z}h_{\alpha_i}$ the \textit{coroot lattice}. For any $\beta=\sum_{i\in I'}c_i\alpha_i\in \mathcal{Q}$, set $h_{\beta}=\sum_{i\in I'}c_ih_{\alpha_i}\in\mathcal{Q}^{\vee}$. We should point out that the set $\{h_{\alpha_i}|i\in I'\}$ is not linearly independent when $m=n$. 
The associated Cartan matrix is defined as the $(N-1)\times (N-1)$-matrix $A=\left(a_{i j}\right)_{i,j\in I'}$ with entries 
\begin{gather*}
    a_{i j}:=\alpha_j(h_{\alpha_i})=\left((-1)^{|i|}+(-1)^{|i+1|}\right) \delta_{i, j}-(-1)^{|i|}\delta_{i, j+1}-(-1)^{|i+1|}\delta_{i+1, j}.
\end{gather*}

\vspace{0.5em}
Let $\mathcal{L}\mathfrak{a}:=\mathfrak{a}\otimes\mathbb{C}[\lambda,\lambda^{-1}]$ be the loop superalgebra associated with the Lie superalgebra $\mathfrak{a}$ (of type $\boldsymbol{A}$). For $x\in\mathfrak{a}$ and $r\in\mathbb{Z}$, define $x^{(r)}=x\otimes \lambda^r$ whose parity is given by $|x^{(r)}|=|x|$. The affine Lie superalgebra $\widehat{\mathfrak{a}}$ of type $\boldsymbol{A}$ is the central extension $\mathcal{L}\mathfrak{gl}(m|n)\oplus\mathbb{C}K$ of the loop superalgebra (without derivation), which satisfies the super-bracket relations:
\begin{align*}
    &[K,\,\widehat{\mathfrak{a}}]=0, \\
    &[x^{(r)}+aK,\,y^{(s)}+bK]=[x,\,y]^{(r+s)}+r\delta_{r,-s}\kappa(x,y)K,~~~x,y\in \mathfrak{a},~r,s\in\mathbb{Z},~a,b\in\mathbb{C},
\end{align*}
where $\kappa$ is a $\mathbb{C}$-value invariant supersymmetric bilinear form on $\mathfrak{a}$. Let $\delta$ be the null root of $\widehat{\mathfrak{a}}$ satisfying $\delta(K)=\delta(\alpha_i^{\vee})=0$, $i\in I'$ and $|\delta|=0$. Set $\alpha_0=\delta+\epsilon_N-\epsilon_1$. Then $\{\alpha_i|i\in I'\cup\{0\}\}$ is the set of simple roots of $\widehat{\mathfrak{a}}$. 

Let $q$ be an indeterminate and set $q_{i}:=q^{(-1)^{|i|}}$. By convention, we use $\mathbb{C}(q)$ to denote the field of rational functions in $q$ over $\mathbb{C}$. For a nonzero $a\in\mathbb{C}(q)$, the $a$-graded commutator is defined as 
$$
\left[x,y\right]_{a}=xy-(-1)^{\mid y\mid\mid x\mid} ayx,\quad \text{for all homogeneous}~~x,y\in\widehat{\mathfrak{a}}. 
$$
In particular, setting $a=1$ recovers the standard supercommutator: $\left[x,\,y\right]_{1}=\left[x,\,y\right]$.

\subsection{Drinfeld presentation $U_q^{D}(\widehat{\mathfrak{gl}(m|n)})$}\label{se:Drinfeldpresentation}

In this subsection, we first review the Drinfeld realizations of the quantum affine superalgebra $U_q({\widehat{\mathfrak{gl}(m|n)}})$ and $U_q({\widehat{\mathfrak{sl}(m|n)}})$.
\begin{defi}\label{Drinfeld def}(\cite{Zyz97})\,
The superalgebra $U_q^{D}(\widehat{\mathfrak{gl}(m|n)})$ is defined as the unital associative superalgebra over $\mathbb{C}(q)$ generated by the central element $q^{\pm\frac{c}{2}}$ and the Drinfeld current generators $ X_{i,n}^{\pm}$, $K_{j,\pm r}^{\pm}$ for $i \in I^{'}, j\in I, n\in \mathbb{Z}$, satisfying relations in terms of the following generating functions in a formal variable $z$ :
\begin{gather*}
    X_{i}^{\pm}(z):=\sum_{k\in \mathbb{Z}}X_{i,k}^{\pm}z^{-k},\quad 
    K_{j}^{\pm}(z):=\sum_{r\in\mathbb{N}}K_{j,\pm r}^{\pm}z^{\mp r}.
\end{gather*}
The grading of the generators are: $|X_i^{ \pm}(z)|=|\alpha_i|=|i|+|i+1|$ and zero otherwise. The defining relations are given by
\begin{align}
K_{j,0}^{+}K_{j,0}^{-}&=K_{j,0}^{-}K_{j,0}^{+}=1,
\\
K_i^{ \pm}(z) K_j^{ \pm}(w)&=K_j^{ \pm}(w) K_i^{ \pm}(z),
\\
K_i^{+}(z) K_i^{-}(w)&=K_i^{-}(w) K_i^{+}(z), \quad i \leqslant m,
\\ \label{kiki}
\frac{w_{-} q-z_{+} q^{-1}}{z_{+} q-w_{-} q^{-1}} K_i^{+}(z)K_i^{-}(w)
&=\frac{w_{+} q-z_{-} q^{-1}}{z_{-} q-w_{+} q^{-1}} K_i^{-}(w) K_i^{+}(z), \quad m<i \leqslant m+n,
\\ \label{kikj}
\frac{z_{ \pm}-w_{\mp}}{z_{ \pm} q-w_{\mp} q^{-1}} K_i^{\mp}(w)^{-1} K_j^{ \pm}(z)
&=\frac{z_{\mp}-w_{ \pm}}{z_{\mp} q-w_{ \pm} q^{-1}} K_j^{ \pm}(z) K_i^{\mp}(w)^{-1}, \quad i>j, \\
\label{kjXi rel2}
K_j^{ \pm}(z)^{-1} X_i^{\epsilon}(w) K_j^{ \pm}(z)&=X_i^{\epsilon}(w), \quad j-i \leqslant-1,
\\ \label{KjXi rel2}
K_j^{ \pm}(z)^{-1} X_i^{\epsilon}(w) K_j^{ \pm}(z)&=X_i^{\epsilon}(w), \quad j-i \geqslant 2,
\\
K_i^{\pm}(z)^{\epsilon} X_i^{\epsilon}(w) K_i^{\pm}(z)^{-\epsilon}&=\frac{z_{\pm \epsilon} q_{i}-w q_{i}^{-1}}{z_{\pm \epsilon}-w} X_i^{\epsilon}(w), \quad i\neq m,
\\
K_{i+1}^{ \pm}(z)^{\epsilon} X_i^{\epsilon}(w) K_{i+1}^{\pm}(z)^{-\epsilon}&=\frac{z_{\pm \epsilon} q_{i}^{-1}-w q_{i}}{z_{\pm \epsilon}-w} X_i^{\epsilon}(w), \quad i\neq m,
\\  \label{kiXm rel1}
K_i^{ \pm}(z)^{\epsilon} X_m^{\epsilon}(w) K_i^{ \pm}(z)^{-\epsilon}&=\frac{z_{\pm \epsilon} q-w q^{-1}}{z_{\pm \epsilon}-w} X_m^{\epsilon}(w), \quad i=m, m+1,
\\ \label{com rels1 Xi}
\left(z q_{i}^{\pm 1}-w q_{i}^{\mp 1}\right) X_i^{\pm}(z) X_i^{\pm}(w)&=\left(z q_{i}^{ \mp 1}-w q_{i}^{\pm 1}\right) X_i^{\pm}(w) X_i^{\pm}(z), \quad i\neq m,
\\
\left[X_m^{ \pm}(z), X_m^{ \pm}(w)\right]&=0,
\\
\left(z-w\right)X_i^{+}(z) X_{i+1}^{+}(w)&=\left(z q_{i+1}-w q_{i+1}^{-1}\right) X_{i+1}^{+}(w) X_i^{+}(z), \quad i \leqslant m+n-1,
\\
\left(z q_{i+1}-w q_{i+1}^{-1}\right) X_i^{-}(z) X_{i+1}^{-}(w)&=(z-w) X_{i+1}^{-}(w) X_i^{-}(z), \quad i \leqslant m+n-1,
\\
 X_{i}^{\epsilon}(z)X_{j}^{\epsilon'}(w)&=X_{j}^{\epsilon'}(w)X_{i}^{\epsilon}(z),\quad |i-j|>1,
\end{align}
\begin{equation}
\begin{aligned}
\left[X_i^{+}(z), X_j^{-}(w)\right]&= -(-1)^{|\alpha_i|}\left(q-q^{-1}\right) \delta_{i j}\left(\delta\left(wz^{-1} q^c\right) K_{i+1}^{+}\left(w_{+}\right) K_i^{+}\left(w_{+}\right)^{-1}\right.
\\
& \quad\left.-\delta\left(wz^{-1} q^{-c}\right) K_{i+1}^{-}\left(z_{+}\right) K_i^{-}\left(z_{+}\right)^{-1}\right),
\end{aligned}
\end{equation}
where $\epsilon,\epsilon'\in\{+,-\}$ and $z_{\pm}=zq^{\pm\frac{c}{2}}$. The following are the Serre relations
\begin{align}\label{serre 1}
&\mathrm{Sym}_{z_1,z_2}\bigl[\!\!\bigl[X_{i}^{\pm}(z_1),\bigl[\!\!\bigl
[X_{i}^{\pm}(z_2),X_{i\pm 1}(w)\bigr]\!\!\bigr]\bigr]\!\!\bigr]=0 &(i\neq m),\\ \label{serre 2}
&\mathrm{Sym}_{z_1,z_2}\bigl[X_{m}^{\pm}(z_1),\bigl[\!\!\bigl[
X_{m+1}^{\pm}(z),\bigl[\!\!\bigl[X_{m}^{\pm}(z_2),X_{m-1}^{\pm}(w)
\bigr]\!\!\bigr]\bigr]\!\!\bigr]\bigr]=0 &(m,n>1),
\end{align}
where $z,w,z_1,z_2$ are formal commutative variables and
$$
\bigl[\!\!\bigl[X,Y\bigr]\!\!\bigr]=\left[X,Y\right]_{q^{-\beta(h_\gamma)}}:=
XY-(-1)^{|X||Y|}q^{-\beta(h_\gamma)}YX
$$
if $K_iXK_i^{-1}=q^{\beta(h_i)}X$ and $K_iYK_i^{-1}=q^{\gamma(h_i)}Y$.
\end{defi}

Let
\begin{align}
\dot{X}_{i}^{+}(z)&=(q-q^{-1})^{-1}X_{i}^{+}(zq^{\nu_i}), \\ 
\dot{X}_{i}^{-}(z)&=(-1)^{|\alpha_i|}(q^{-1}-q)^{-1}X_{i}^{-}(zq^{\nu_i}),\\ \dot{K}_{i}^{\pm}(z)&=K_{i+1}^{\pm}(zq^{\nu_i}){K_{i}^{\pm}(zq^{\nu_i})}^{-1},
\end{align}
where $\nu_i=\sum_{j=1}^{i}(-1)^{|j|}$. The sub-superalgebra of $U_q^{D}({\widehat{\mathfrak{gl}(m|n)}})$, which is generated by $q^{\pm\frac{c}{2}}$, $\dot{K}_i:=K_{i+1,0}K_{i,0}^{-1}$ and the coefficients of the entries of the matrices $\dot{X}_{i}^{\pm}(z)$, $\dot{K}_{i}^{\pm}(z)$ for $i\in I'$ coincides with the Drinfeld realization of $U_q^{D}({\widehat{\mathfrak{sl}(m|n)}})$ refer to \cite{BCFK22}, see also \cite{LYZ24}.  

Let $\{\alpha_i\}_{i=1}^{m+n-1}$ be the standard simple positive roots of $\widehat{\mathfrak{sl}}$, and let $\Delta^+=\{\alpha_j+\alpha_{j+1}+\ldots+\alpha_i\}_{1\leqslant j\leqslant i<m+n-1}$ be the set of positive roots. A total order ''$\leq$'' on $\Delta^+$ is defined by the condition \cite{BCFK22}:
\begin{equation}\label{order}
\alpha_j+\alpha_{j+1}+\ldots+\alpha_i\leqslant\alpha_{j'}+\alpha_{j'+1}+\ldots+\alpha_{i'},
\ \, \mathrm{iff}\ \, j<j'\ \mathrm{or}\ j=j',i\leqslant i'.
\end{equation}
For each $\beta\in \Delta^+$, let $\preceq_\beta$ be a total order on $\mathbb{Z}$. 
This allows us to define a total order on $\Delta^+\times \mathbb{Z}$ as follows:
\begin{equation}\label{extended order}
(\beta,r)\leqslant (\beta',r'), \ \, \mathrm{iff}\ \,
\beta<\beta' \ \mathrm{or}\ \beta=\beta', r\preceq_\beta r'.
\end{equation}

Let \textbf{M} denote the set of all functions \textbf{m}: $\Delta^+\times \mathbb{Z}\to \mathbb{N}$ equipped with the lexicographical order induced by the total order on $\Delta^+\times \mathbb{Z}$. Explicitly, for any \textbf{m}$_{1}$, \textbf{m}$_{2}\in$ \textbf{M}, let $(\alpha_{0},k_{0})=\mathrm{~min~}\{(\alpha,k)\in\Delta^+\times \mathbb{Z}\mid\mathrm{\textbf{m}}_{1}(\alpha,k)\neq\mathrm{\textbf{m}}_{2}(\alpha,k)\}$, then
$$
\mathrm{\textbf{m}}_{1}<\mathrm{\textbf{m}}_{2}, \mathrm{~if}\ \ \mathrm{\textbf{m}}_{1}(\alpha_{0},k_{0})<\mathrm{\textbf{m}}_{2}(\alpha_{0},k_{0}).
$$
This lexicographical order on \textbf{M} induces a natural total order on the set of monomials
$$
x_{\mathrm{\textbf{m}}}^{+}:=\prod\limits_{(\alpha,r)\in \Delta^+\times \mathbb{Z}}^{\rightarrow}(x_{\alpha,r})^{\mathrm{\textbf{m}}(\alpha,r)},\qquad
x_{\mathrm{\textbf{m}}}^{-}:=\prod\limits_{(\alpha,r)\in \Delta^+\times \mathbb{Z}}^{\leftarrow}(x_{-\alpha,r})^{\mathrm{\textbf{m}}(\alpha,r)},
$$
where the arrow $\rightarrow$ indicates that the product is taken in increasing order from left to right with respect to the order on $\Delta^+\times \mathbb{Z}$, and $\leftarrow$ indicates the opposite(decreasing) order.

We recall the Drinfeld type PBW basis constructed in \cite{Ts21}. For this we consider the total orders on $\Delta$ and on $\Delta^+\times \mathbb{Z}$. For each $(\alpha,r)\in\Delta^+\times \mathbb{Z}$, write $\alpha=\alpha_{i_{1}}+\cdots+\alpha_{i_{p}}$, and fix:
\begin{enumerate}
  \item a decomposition $r=r_{1}+\cdots+r_{p}$, $r_{i}\in\mathbb{Z}$;
  \item a sequence $(q_{1},\cdots,q_{p-1})\in\{q,q^{-1}\}^{p-1}$.
\end{enumerate}
Now, each $(\alpha,r)\in\Delta^+\times \mathbb{Z}$ defines a vector $X_{\pm\alpha,r}\in U_{q}^{\pm}$ as
$$
X_{\pm\alpha,r}:=[\cdots[[X_{i_{1},r_{1}}^{\pm},X_{i_{2},r_{2}}^{\pm}]_{q_{1}},
X_{i_{3},r_{3}}^{\pm}]_{q_{2}},\cdots,X_{i_{p},r_{p}}^{\pm}]_{q_{p-1}}.
$$

For every \textbf{m}$\in$ \textbf{M} satisfying \textbf{m}$(\alpha,k)\leqslant 1$ for $|\alpha|=1$, we define the monomials
$$
X_{\mathrm{\textbf{m}}}^{+}:=\prod\limits_{(\alpha,r)\in \Delta^+\times \mathbb{Z}}^{\rightarrow}(X_{\alpha,r})^{\mathrm{\textbf{m}}(\alpha,r)},\qquad
X_{\mathrm{\textbf{m}}}^{-}:=\prod\limits_{(\alpha,r)\in \Delta^+\times \mathbb{Z}}^{\leftarrow}(X_{-\alpha,r})^{\mathrm{\textbf{m}}(\alpha,r)}.
$$

Finally, we denote by \textbf{H} the set of all functions \textbf{h}: $I\times\mathbb{Z}\rightarrow\mathbb{Z}$ that have finite support and satisfy \textbf{h}$(i,r)\geq0$ for $r\neq0$. For any \textbf{h}$\in$ \textbf{H}, we introduce the monomials
$$
K_{\mathrm{\textbf{h}}}^{\pm}:=\prod\limits_{(i,r)\in I\times \mathbb{Z}}(K_{i,\pm r}^{\pm})^{\mathrm{\textbf{h}}(i,\pm r)},\qquad
K_{\mathrm{\textbf{h}}}^{0}:=\prod\limits_{i\in I}(K_i)^{\mathrm{\textbf{h}}(i,0)},\qquad K_{\mathrm{\textbf{h}}}:=K_{\mathrm{\textbf{h}}}^{-}K_{\mathrm{\textbf{h}}}^{0}K_{\mathrm{\textbf{h}}}^{+}.
$$

Given that $U_q^{D}(\widehat{\mathfrak{g}})$ and $U_q^{D}(\widehat{\mathfrak{sl}(m|n)})$ are identical up to a central element, the following result can be readily derived by arguments analogous to those in Theorem 4.2 of \cite{BCFK22}.

\begin{theo}\label{UD basis}
The set of monomials
$$
\{X_{\mathrm{\textbf{m}}}^{-}K_{\mathrm{\textbf{h}}}q^{\frac{rc}{2}}X_{\mathrm{\textbf{m}^{'}}}^{+}\mid\mathrm{\textbf{m}},\mathrm{\textbf{m}^{'}}\in\mathrm{\textbf{M}},
\mathrm{\textbf{h}}\in\mathrm{\textbf{H}},r\in\mathbb{Z}\}
$$
forms a linear basis of $U_{q}^{D}(\widehat{\mathfrak{g}})$.
\end{theo}

\subsection{$R$-matrix presentation $U_{q}^{R}(\widehat{\mathfrak{gl}(m|n)})$}

The $R$-matrix $R(z)$ is an element of ${\rm End~} \mathcal{V}^{\otimes 2}$ defined by
\begin{equation}
\begin{aligned}\label{R-matrix}
R(z)&{=\sum_{i=1}^{m} E_{ii}\otimes E_{ii}+\frac{q-zq^{-1}}{zq-q^{-1}}\sum_{i=m+1}^{m+n}E_{ii}\otimes E_{ii}+\frac{z-1}{zq-q^{-1}}\sum_{i\neq j}E_{ii}\otimes E_{jj}}
\\
&{+\frac{z(q-q^{-1})}{zq-q^{-1}}\sum_{i< j}(-1)^{|j|}E_{ij}\otimes E_{ji}+\frac{q-q^{-1}}{zq-q^{-1}}\sum_{i> j}(-1)^{|j|}E_{ij}\otimes E_{ji}},
\end{aligned}
\end{equation}
where $z$ is a formal variable. 
The following relations hold for  $R_{21}(z)=P_{12} R_{12}(z) P_{12}$:
\begin{equation}\label{R}
R_{12}(z) R_{21}\left(z^{-1}\right)=1, \quad~R_{21}\left(z^{-1}\right)=R_{q^{-1}}(z).
\end{equation}
Recall the RTT realization of the quantum affine superalgebra $U_q({\widehat{\mathfrak{gl}(m|n)}})$ with a slight modification \cite{JLZ25}. 
\begin{defi}
The superalgebra $U_q^{R}(\widehat{\mathfrak{gl}(m|n)})$ is an associative superalgebra over $\mathbb{C}(q)$ generated by the elements ${l^{\pm}_{ij}}^{(r)}$, where $1\leqslant i, j\leqslant N$ and $r\in \mathbb{N}$. Define the generating matrices
\begin{equation*}
\begin{split}
L^{\pm}(z)= \sum_{i,j\in I}E_{ij}\otimes l^{\pm}_{ij}(z)\quad \text{with}\quad l^{\pm}_{ij}(z)= \sum_{r=0}^{\infty} {l^{\pm}_{ij}}^{(r)}z^{\mp r}.
\end{split}
\end{equation*}
The defining relations of generators are given by:
\begin{align} \label{RLL-0}
{l^+_{ji}}^{(0)}={l^-_{ij}}^{(0)}=&0,~~~~1\leqslant i<j\leqslant N, \\ \label{RLL-1}
l{^-_{ii}}^{(0)} {l^+_{ii}}^{(0)}={l^+_{ii}}^{(0)}{l^-_{ii}}^{(0)}=&1,~~~~1\leqslant i\leqslant N, \\ \label{RLL}
R_{12}(zw^{-1})L^{\pm}_{1} (z) L^{\pm}_{2} (w)=& L^{\pm}_{2} (w) L^{\pm}_{1} (z)R_{12}(zw^{-1}),\\ \label{RLL cros}
R_{12}(z_{+}w_{-}^{-1}) L^{+}_{1} (z) L^{-}_{2} (w)=& L^{-}_{2} (w) L^{+}_{1} (z)R_{12}(z_{-}w_{+}^{-1}),
\end{align}
where $z_{ \pm}=z q^{ \pm \frac{c}{2}}$ as in Section \ref{se:Drinfeldpresentation}, and
\begin{gather*}
    L_1^{ \pm}(z)=\sum_{i,j\in I} E_{ij}\otimes 1 \otimes l_{ij}^{\pm}(z),\qquad L_2^{ \pm}(z)=\sum_{i,j\in I}1\otimes E_{ij}\otimes l_{ij}^{\pm}(z).
\end{gather*}
The expansion direction of $R\left(z/w\right)$ can be chosen in $z/w$ or $w/z$ in \eqref{RLL}
and only in $z/w$ in \eqref{RLL cros}.
\end{defi}

\begin{rema}
    When taking $q^{\frac{c}{2}}=1$, the superalgebra $U_{q^{-1}}^R\big(\widehat{\mathfrak{gl}(m|n)}\big)$ coincides with the definition introduced in \cite[Definition 3.1]{Zhf16} by setting (see also \cite{LWZ24,LZ25})
    $$S(z)=L^{+}(z^{-1}),\quad T(z)=L^-(z^{-1}).$$
\end{rema}

\vspace{0.5em}
Let $L_{j}^{i}$ be the $(i,j)$-entry of the matrix $L$, and set $R_{jt}^{is}=(R_1)_{j}^{i}\otimes(R_{2})_{t}^{s}$ with $R=R_{1}\otimes R_{2}$. 
The actions of $R(z)$ and $L^{ \pm}(z)$ on the tensor product $V\otimes V$ are given by
\begin{align*}
    R(z)\left(v_{a^{\prime}} \otimes v_{b^{\prime}}\right)
&=R(z)_{a^{\prime} b^{\prime}}^{ab}\left(v_a \otimes v_b\right), \\
L^{\pm}(z) v_{a^{\prime}}&=L^{\pm}(z)_{a^{\prime}}^{a} v_a.
\end{align*}
In the above, we employ the summation convention over repeated indices, i.e., the first equation stands for
$$
R(z)\left(v_{a^{\prime}} \otimes v_{b^{\prime}}\right)=\sum_{a, b}R(z)_{a^{\prime} b^{\prime}}^{ab}\left(v_a \otimes v_b\right).
$$

The graded tensor product rule incorporates extra signs into the matrix form of the relations:
\begin{align*}
&R\left(zw^{-1}\right)_{a^{\prime \prime} b^{\prime\prime}}^{a b}L^{ \pm}(z)_{a^{\prime}}^{a^{\prime \prime} } L^{\pm}(w)_{b^{\prime}}^{b^{\prime \prime} }(-1)^{|a^{\prime}|\left(|b^{\prime}|+
|b^{\prime \prime}|\right)}
=L^{\pm}(w)_{b^{\prime\prime}}^{b} L^{\pm(z)_{a^{\prime \prime}}^{a}} R\left(zw^{-1}\right)_{a^{\prime} b^{\prime}}^{a^{\prime\prime}b^{\prime\prime}}(-1)^{|a|\left(|b|+|b^{\prime \prime}|\right)},
\\		
&R\left(z_{+}w_{-}^{-1}\right)_{a^{\prime \prime} b^{\prime \prime}}^{a b} L^{+}(z)_{a^{\prime}}^{a^{\prime \prime}} L^{-}(w)_{b^{\prime}}^{b^{\prime \prime}}(-1)^{|a^{\prime}|
\left(|b^{\prime}|+|b^{\prime \prime}|\right)}
=L^{-}(w)_{b^{\prime \prime}}^{b} L^{+}(z)_{a^{\prime \prime}}^{a} R\left(z_{-}w_{+}^{-1}\right)_{a^{\prime} b^{\prime}}^{a^{\prime \prime} b^{\prime \prime}}(-1)^{|a|\left(|b|+|b^{\prime \prime}|\right)}.
\end{align*}
We introduce the matrix $\theta$ defined as
$$
\theta_{a^{\prime} b^{\prime}}^{ab}
=(-1)^{|a||b|}\delta_{a a^{\prime}} \delta_{b b^{\prime}}.
$$
Using this matrix,  the relation \eqref{RLL} and \eqref{RLL cros} can be rewritten as
\begin{align*}
&R\left(zw^{-1}\right) L_1^{ \pm}(z) \theta L_2^{\pm}(w) \theta=\theta L_2^{ \pm}(w) \theta L_1^{ \pm}(z) R\left(zw^{-1}\right),
\\ 
&R\left(z_{+}w_{-}^{-1}\right) L_1^{+}(z) \theta L_2^{-}(w) \theta=\theta L_2^{-}(w) \theta L_1^{+}(z) R\left(z_{-}w_{+}^{-1}\right).
\end{align*}
	
Note that the tensor products above don't necessarily care about gradings. From the $RLL$ relations, it is straightforward to derive the following matrix equations:
\begin{align}\label{RLL2.1}
& R_{21}\left(zw^{-1}\right) \theta_{12} L_2^{ \pm}(z) \theta_{12} L_1^{ \pm}(w)=L_1^{ \pm}(w) \theta_{12} L_2^{ \pm}(z) \theta_{12} R_{21}\left(zw^{-1}\right),
\\ 	\label{RLL2.2}
& R_{21}\left(z_{-}w_{+}^{-1}\right) \theta_{12} L_2^{-}(z) \theta_{12} L_1^{+}(w)
=L_1^{+}(w) \theta_{12} L_2^{-}(z) \theta_{12} R_{21}\left(z_{+}w_{-}^{-1}\right),
\\
&\theta_{12} L_2^{ \pm}(z)^{-1} \theta_{12} L_1^{ \pm}(w)^{-1} R_{21}\left
(zw^{-1}\right)=R_{21}\left(zw^{-1}\right) L_1^{ \pm}(w)^{-1} \theta_{12} L_2^{ \pm}(z)^{-1} \theta_{12} ,
\\ \label{RLL2.4}
&\theta_{12} L_2^{+}(z)^{-1} \theta_{12} L_1^{-}(w)^{-1} R_{21}\left(z_{+}
w_{-}^{-1}\right)=R_{21}\left(z_{-}w_{+}^{-1}\right) L_1^{-}(w)^{-1} \theta_{12} L_2^{+}(z)^{-1} \theta_{12},
 \\ \label{RLL2.5}
& L_1^{ \pm}(w)^{-1} R_{21}\left(zw^{-1}\right) \theta_{12} L_2^{ \pm}(z) \theta_{12}=\theta_{12} L_2^{ \pm}(z) \theta_{12} R_{21}\left(zw^{-1}\right) L_1^{ \pm}(w)^{-1} ,
\\ \label{RLL2.6}
& L_1^{-}(w)^{-1} R_{21}\left(z_{+}w_{-}^{-1}\right) \theta_{12} L_2^{+}(z)\theta_{12}=\theta_{12} L_2^{+}(z) \theta_{12} R_{21}\left(z_{-}w_{+}^{-1}\right) L_1^{-}(w)^{-1},
\\ \label{RLL2.7}
& L_1^{+}(w)^{-1} R_{21}\left(z_{-}w_{+}^{-1}\right) \theta_{12}L_2^{-}(z) \theta_{12}=\theta_{12} L_2^{-}(z)\theta_{12} R_{21}\left(z_{+}w_{-}^{-1}\right) L_1^{+}(w)^{-1}.
\end{align}

\vspace{1em}
A PBW basis of $U_q^{R}(\widehat{\mathfrak{gl}(m|n)})$ was established in \cite{LZ25}.
\begin{prop}\label{base}
Let $\mathfrak{B}$ denote the set of all ordered monomials
\begin{equation}\label{base:Uq}
\begin{aligned}
\prod_{1-N\leq k\leq -1}^{\rightarrow}&\prod_{1-k\leq i\leq N}^{\rightarrow}\left\{\left({l_{i,i+k}^{-}}^{(0)}\right)^{m_{i,i+k,0}}
\left({l_{i,i+k}^{-}}^{(1)}\right)^{m_{i,i+k,1}}\left({l_{i,i+k}^{+}}^{(1)}\right)
^{\widetilde{m}_{i,i+k,1}}\cdots\right\}\\
&\times\prod_{1\leq i\leq N}^{\rightarrow}\left\{\left({l_{ii}^{-}}^{(0)}\right)
^{m_{i,i,0}}\left({l_{ii}^{+}}^{(0)}\right)^{\widetilde{m}_{i,i,0}}
\left({l_{ii}^{-}}^{(1)}\right)^{m_{i,i,0}}\left({l_{ii}^{+}}^{(1)}\right)
^{\widetilde{m}_{i,i,1}}\cdots\right\}\\
&\times\prod_{1\leqslant k\leqslant N-1}^{\rightarrow}\prod_{1\leqslant i\leqslant k}^{\rightarrow}\left\{\left({l_{i,i+k}^{+}}^{(0)}\right)^{\widetilde{m}_{i,i+k,0}}
\left({l_{i,i+k}^{-}}^{(1)}\right)^{m_{i,i+k,1}}\left({l_{i,i+k}^{+}}^{(1)}\right)
^{\widetilde{m}_{i,i+k,1}}\cdots\right\}
\end{aligned}
\end{equation}
subject to the following conditions on the exponents:
\begin{align}\label{base:1}
&m_{i,j;r},\ \widetilde{m}_{i,j;r}\in\mathbb{Z}_{\geqslant 0}\quad \text{for }|i|+|j|=\bar{0},
\\ \label{base:2}
&m_{i,j;r},\ \widetilde{m}_{i,j;r}\in\{0,1\}\quad \text{for }|i|+|j|=\bar{1},
\\ \label{base:3}
&m_{i,i;0}\widetilde{m}_{i,i;0}=0\quad\text{for }i\in I.
\end{align}
Then the monomial set $\mathfrak{B}$ forms an ordered basis of $U_q^{R}(\widehat{\mathfrak{gl}(m|n)})$ over $\mathbb{C}(q)$.
\end{prop}

\subsection{Gauss decomposition of $L^{\pm}(z)$}


Consider an $N\times N$ matrix $X=(x_{ij})_{ij}$ over a unital ring. For indices $i,j\in I$, let $X^{ij}$ be the submatrix obtained by removing the $i$-th row and $j$-th column of $X$, and assume that $X^{ij}$ is invertible.
\begin{defi} (\cite{GGRW05})
The $(i,j)$-th quasideterminant of $X$ is given by the formula
$$
|X|_{i j}=x_{ij}-r_{i}^{j}(X^{ij})^{-1}c_{j}^{i}.
$$
Here, $r_{i}^{j}$ denotes the row vector from the $i$-th row of $X$ with $x_{ij}$ deleted, and $c_{i}^{j}$ denotes the column vector from the $j$-th column of $X$ with $x_{ij}$ deleted.
\end{defi}
The quasideterminant $|X|_{i j}$ has a graphical notation where the entry $x_{ij}$ is highlighted with a box:
$$
|X|_{i j}=\left|\begin{array}{ccccc}
x_{11} & \cdots & x_{1 j} & \cdots & x_{1 n} \\
\cdots & & \cdots & \\
x_{i 1} & \cdots & \boxed{x_{i j}} & \cdots & x_{i n} \\
\cdots & & \cdots & \\
x_{n 1} & \cdots & x_{n j} & \cdots & x_{n n}
\end{array}\right|.
$$

Now we introduce the Gaussian generators for the superalgebra $U_q^{R}(\widehat{\mathfrak{gl}(m|n)})$. For $i<j$, define
\begin{align*}
k_i^{\pm}(z)&=\left|\begin{array}{cccc}
l_{11}^{\pm}(z) & \cdots & l_{1, i-1}^{\pm}(z) & l_{1 i}^{\pm}(z) \\
\vdots & \ddots & & \vdots \\
l_{i 1}^{\pm}(z) & \cdots & l_{i, i-1}^{\pm}(z) & \boxed{l_{i i}^{\pm}(z)}
\end{array}\right|, 
\end{align*}
\begin{align*}
f_{i j}^{\pm}(z)&={k_i^{\pm}(z)}^{-1}\left|\begin{array}{cccc}
l_{11}^{\pm}(z) & \cdots & l_{1, i-1}^{\pm}(z) & l_{1 j}^{\pm}(z) \\
\vdots & \ddots & \vdots & \vdots \\
l_{i-1, i}^{\pm}(z) & \cdots & l_{i-1, i-1}^{\pm}(z) & l_{i-1, j}^{\pm}(z) \\
l_{i 1}^{\pm}(z) & \cdots & l_{i, i-1}^{\pm}(z) & \boxed{l_{i j}^{\pm}(z)}
\end{array}\right|, 
\end{align*}
\begin{align*}
e_{j i}^{\pm}(z)&=\left|\begin{array}{cccc}
l_{11}^{\pm}(z) & \cdots & l_{1, i-1}^{\pm}(z) & l_{1 i}^{\pm}(z) \\
\vdots & \ddots & \vdots & \vdots \\
l_{i-1,1}^{\pm}(z) & \cdots & l_{i-1, i-1}^{\pm}(z) & l_{i-1, i}^{\pm}(z) \\
l_{j i}^{\pm}(z) & \cdots & l_{j, i-1}^{\pm}(z) & \boxed{l_{j i}^{\pm}(z)}
\end{array}\right|{k_i^{\pm}(z)}^{-1} .
\end{align*}

The Gauss decomposition remains valid in our setting.
\begin{prop}(\cite{FHS97})
The generating matrices $L^{\pm}(z)$ have the unique Gauss decomposition:
\begin{equation}\label{Gauss_dec}
\begin{aligned}
L^{\pm}(z)&=\left(\begin{array}{cccc}
1 & & & 0 \\
e_{21}^{\pm}(z) & \ddots & & \vdots \\
\vdots & & \ddots \\
e_{m+n, 1}^{\pm}(z) & e_{m+n, 2}^{\pm}(z) & \cdots & 1
\end{array}\right)
\left(\begin{array}{cccc}
k_1^{\pm}(z) & & \cdots & 0 \\
& k_2^{\pm}(z) & & \vdots \\
\vdots & & \ddots & \\
0 & \cdots & & k^{\pm}_{m+n}(z)
\end{array}\right)\\
&\hspace{3em}\times\left(\begin{array}{cccc}
1 & f_{12}^{\pm}(z) & \cdots & f_{1, m+n}^{\pm}(z) \\
& \ddots & & f_{2, m+n}^{\pm}(z) \\
& & \ddots & \vdots \\
0 & & & 1
\end{array}\right),
\end{aligned}
\end{equation}
where the coefficients of $e_{ji}^{\pm}(z)$, $f_{ij}^{\pm}(z)$ and $k_{i}^{\pm}(z)$, $1\leqslant i,j \leqslant m+n$ with $i<j$, belong to the superalgebra $U_q^{R}(\widehat{\mathfrak{gl}(m|n)})$ and the $k_{i}^{\pm}(z)$ are invertible.
\end{prop}
More precisely, the entries of $L^{\pm}(z)$ satisfy the following relations:
\begin{align}\label{Gauss ii}
l_{ii}^{\pm}(z)&= k_{i}^{\pm}(z) + \sum_{s < i} e_{is}^{\pm}(z) k_{s}^{\pm}(z) f_{si}^{\pm}(z),\\ \label{Gauss ij}
l_{ij}^{\pm}(z)&= k_{i}^{\pm}(z) f_{ij}^{\pm}(z) + \sum_{s< i} e_{is}^{\pm}(z) k_{s}^{\pm}(z) f_{sj}^{\pm}(z),\\ \label{Gauss ji}
l_{ji}^{\pm}(z)&= e_{ji}^{\pm}(z) k_{i}^{\pm}(z) + \sum_{s <i} e_{js}^{\pm}(z) k_{s}^{\pm}(z) f_{si}^{\pm}(z).
\end{align}

The inverse matrix $L^{ \pm}(z)^{-1}$ admits the following Gauss decomposition:
\begin{align}\notag
L^{\pm}(z)^{-1}&= \left(\begin{array}{cccc}
			1 & -f_1^{ \pm}(z) & \cdots & \\
			\vdots & & \ddots & \vdots \\
			& & & -f_{N-1}^{ \pm}(z) \\
			0 & \cdots & & 1
		\end{array}\right)\left(\begin{array}{ccc}
			k_1^{ \pm}(z)^{-1} & \cdots & 0 \\
			\vdots & \ddots & \vdots \\
			0 & \cdots & k_N^{ \pm}(z)^{-1}
		\end{array}\right) \\ \label{e:Gauss}
		&\hspace{3em} \times\left(\begin{array}{cccc}
			1 & \cdots & & 0 \\
			-e_1^{ \pm}(z) &  & & \\
			\vdots & \ddots & & \vdots \\
			&\cdots & -e_{N-1}^{ \pm}(z) & 1
		\end{array}\right).
\end{align}

Throughout the following, we use the abbreviated notation $e_{i}^{\pm}(z):=e_{i + 1,i}^{\pm}(z)$ and $f_{i}^{\pm}(z):=f_{i,i+ 1}^{\pm}(z)$. Furthermore, we introduce the coefficient expansions:
\begin{align*}
k^{\pm}_{i}(z)&= \sum_{r=0}^{\infty} {k^{\pm}_{i,\pm r}}z^{\mp r};\\
e^{+}_{i}(z)&= \sum_{r=1}^{\infty} {e^{+}_{i}}^{(r)}z^{- r},~~e^{-}_{i}(z)= \sum_{r=0}^{\infty} {e^{-}_{i}}^{(r)}z^{r};\\
f^{+}_{i}(z)&= \sum_{r=0}^{\infty} {f^{+}_{i}}^{(r)}z^{- r},~~
f^{-}_{i}(z)= \sum_{r=1}^{\infty} {f^{-}_{i}}^{(r)}z^{r}.
\end{align*}

Now we state one of the main results for this paper.
\begin{theo}\label{main theo}
There exists a $\mathbb{C}(q)$-superalgebra isomorphism $\varphi${\rm :} $U_{q}^{D}(\widehat{\mathfrak{gl}(m|n)})\rightarrow U_q^{R}(\widehat{\mathfrak{gl}(m|n)})$ defined as follows{\rm :}
\begin{align}\label{X+}
K_{j}^{\pm}(z)&\mapsto k_{j}^{\pm}(z),\\
X_{i}^{+}(z)&\mapsto x_{i}^{+}(z):=e_{i}^{+}(z_{-})-e_{i}^{-}(z_{+}),\\ \label{X-}
X_{i}^{-}(z)&\mapsto x_{i}^{-}(z):=f_{i}^{+}(z_{+})-f_{i}^{-}(z_{-}),
\end{align}
where $i=1,\cdots,N-1$ and $j=1,\cdots,N$.
\end{theo}
This theorem will be proved in Section \ref{se:isomorphismtheorem}.

\section{Proof of Theorem \ref{main theo} }\label{se:isomorphismtheorem}
We divide our proof into three steps:
\begin{enumerate}
    \item[(1)] Firstly, we give a series of homomorphisms of quantum affine superalgebras, which will be used in the subsequent proofs.
    \item[(2)] Secondly, we formulate the necessary relations of Gaussian generators to be consistent with those of the Drinfeld current generators. Exactly, $\varphi$ is a surjective homomorphism. 
    \item[(3)] Lastly, we show the injectivity of $\varphi$ in terms of the $\mathbb{A}$-form approach. 
\end{enumerate}

\subsection{Homomorphisms between quantum affine superalgebras}
\begin{lemm}(\cite[Lemma 2.3]{JLZ25})\label{hom1}
The following mapping $\omega_{m|n}$ defines an anti-isomorphism of~ $U_q^{R}(\widehat{\mathfrak{gl}(m|n)})$
\begin{equation*}
\omega_{m|n}(L^{\pm}(z))= L^{\pm}(z)^{-1}.
\end{equation*}
\end{lemm}

Using relation \eqref{R}, we multiply both sides of \eqref{RLL} by the inverse of $R\left(\frac{z}{w}\right)$ from the left and the right, we obtain
\begin{equation*}
R\left(zw^{-1}\right)L_{1}^{\pm}(z)^{-1}L_{2}^{\pm}(w)^{-1}=L_{2}^{\pm}(w)^{-1}
L_{1}^{\pm}(z)^{-1}R\left(zw^{-1}\right).
\end{equation*}
Let $\check{L}^{\pm}(z)$ be the generator matrix of $U_{q}^{R}({\widehat{\mathfrak{gl}(n|m)}})$. We then have the following isomorphism.
\begin{lemm}(\cite[Lemma 2.4]{JLZ25})\label{hom2}
The mapping $\rho_{m|n}:U_{q}^{R}({\widehat{\mathfrak{gl}(m|n)}})\rightarrow U_{q}^{R}({\widehat{\mathfrak{gl}(n|m)}})$ defined by
\begin{equation*}
\rho_{m|n}(l_{ij}^{\pm}(z))=\check{l}_{m+n+1-i,m+n+1-j}^{\mp}\left(z^{-1}\right)
\end{equation*}
is an algebra isomorphism. Here, $\check{l}_{ij}^{\pm}(z)$ denotes the corresponding generator of~ $U_{q}^{R}({\widehat{\mathfrak{gl}(n|m)}})$.
\end{lemm}

Let the inverses of $L^{\pm}(z)$ and $\check{L}^{\pm}(z)$ be denoted by
$$
L^{\pm}(z)^{-1}=\left(\left(l_{ij}^{\pm}(z)\right)^{'}\right)_{i,j=1}^{m+n},
\quad \check{L}^{\pm}(z)^{-1}=\left(\left(\check{l}_{ij}^{\pm}(z)\right)^{'}\right)_{i,j=1}^{m+n}.
$$
By Lemma \ref{hom1} and \ref{hom2}, We can get the following algebra isomorphism.
\begin{prop}\label{m to n}
Let $\zeta_{m|n}:U_{q}^{R}({\widehat{\mathfrak{gl}(m|n)}})\rightarrow U_{q}^{R}({\widehat{\mathfrak{gl}(n|m)}})$ be the algebra isomorphism given by $\zeta_{m|n}=\rho_{m|n}\circ \omega_{m|n}$. That is,
$$
\zeta_{m|n}: l_{ij}^{\pm}(z)\mapsto \left(\check{l}_{m+n+1-i,m+n+1-j}^{\mp}\left(z^{-1}\right)\right)^{'}
$$
 Then, for $1\leq i\leq m+n$ and $1\leq j\leq m+n-1$, we have
\begin{equation*}
k_{i}^{\pm}(z)\mapsto \check{k}_{m+n+1-i}^{\mp}\left(z^{-1}\right)^{-1},\quad
e_{j}^{\pm}(z)\mapsto -\check{f}_{m+n-j}^{\mp}\left(z^{-1}\right), \quad
f_{j}^{\pm}(z)\mapsto -\check{e}_{m+n-j}^{\mp}\left(z^{-1}\right),
\end{equation*}
In the above, $\check{k}_{i}^{\pm}(z)$, $\check{e}_{j}^{\pm}(z)$, $\check{f}_{j}^{\pm}(z)$ are the corresponding Gaussian generators of $U_{q}^{R}({\widehat{\mathfrak{gl}(n|m)}})$.
\end{prop}
\begin{proof}
According to \eqref{e:Gauss}, we derive the following expressions:
\begin{align*}
\left(\check{l}_{ii}^{\pm}(z)\right)^{'}&= \check{k}_{i}^{\pm}(z)^{-1} + \sum_{u >i} \left(\check{f}_{iu}^{\pm}(z)\right)^{'} \check{k}_{u}^{\pm}(z)^{-1} \left(\check{e}_{ui}^{\pm}(z)\right)^{'},\\
\left(\check{l}_{ij}^{\pm}(z)\right)^{'}&= \left(\check{f}_{ij}^{\pm}(z)\right)^{'} \check{k}_{j}^{\pm}(z)^{-1} + \sum_{u> j} \left(\check{f}_{iu}^{\pm}(z)\right)^{'} \check{k}_{u}^{\pm}(z)^{-1} \left(\check{e}_{uj}^{\pm}(z)\right)^{'},\\
\left(\check{l}_{ji}^{\pm}(z)\right)^{'}&= \check{k}_{j}^{\pm}(z)^{-1} \left(\check{e}_{ji}^{\pm}(z)\right)^{'} + \sum_{k >j} \left(\check{f}_{ju}^{\pm}(z)\right)^{'} \check{k}_{u}^{\pm}(z)^{-1} \left(\check{e}_{ui}^{\pm}(z)\right)^{'},
\end{align*}
where
$$
\left(\check{f}_{ij}^{\pm}(z)\right)^{'} = \sum_{i=i_{0}< i_{1} < \ldots < i_{s} = j}
(-1)^{s} \check{f}_{i_{0}i_{1}}^{\pm}(z) \check{f}_{i_{1}i_{2}}^{\pm}(z) \cdots \check{f}_{i_{s-1}i_{s}}^{\pm}(z)
$$
and
$$
\left(\check{e}_{ji}^{\pm}(z)\right)^{'} = \sum_{i=i_{0} < i_{1} < \ldots < i_{s} =j}
(-1)^{s} \check{e}_{i_{s}i_{s-1}}^{\pm}(z) \cdots \check{e}_{i_{2}i_{1}}^{\pm}(z) \check{e}_{i_{1}i_{0}}^{\pm}(z).
$$
Applying equations \eqref{Gauss ii}--\eqref{Gauss ji}, we immediately find: $\zeta_{m|n}(k_{1}^{\pm}(z)) = \check{k}_{m+n}^{\mp}\left(z^{-1}\right)^{-1}$,
$\zeta_{m|n}\left( f_{1j}^{\pm}(z)\right) =\left(\check{e}_{m+n, m+n+1-j}^{\mp}\left(z^{-1}\right)\right)^{'} $, and $\zeta_{m|n}\left(e_{j1}^{\pm}(z)\right) = \left(\check{f}_{m+n+1-j, m+n}^{\mp}\left(z^{-1}\right)\right)^{'}$.  By
induction on $i$, we derive:
\begin{align*}
\zeta_{m|n}\left(k_{i}^{\pm}(z)\right)&= \check{k}_{m+n+1-i}^{\mp}\left(z^{-1}\right)^{-1},\\
\zeta_{m|n}\left( f_{ij}^{\pm}(z)\right)&=\left(\check{e}_{m+n+1-i, m+n+1-j}^{\mp}\left(z^{-1}\right)\right)^{'} ,\\
\zeta_{m|n}\left(e_{ji}^{\pm}(z)\right)&=\left(\check{f}_{m+n+1-j, m+n+1-i}^{\mp}\left(z^{-1}\right)\right)^{'}.
\end{align*}
The result stated in the proposition corresponds to the special case $j=i+1$.
\end{proof}

For any $p\geqslant 0$ we introduce the algebra homomorphism
\begin{equation*}
\phi_p: U_q^{R}(\widehat{\mathfrak{gl}(m|n)})\rightarrow  U_q^{R}(\widehat{\mathfrak{gl}(m+p|n)}),
\end{equation*}
which takes $l^{\pm}_{ij}(z)$ to $l^{\pm}_{p+i,p+j}(z)$. Consider the composition
\begin{equation*}
\psi_p=\omega_{p+m|n}\circ\phi_p\circ\omega_{m|n}.
\end{equation*}

\begin{lemm}(\cite[Lemma 4.7]{JLZ25})\label{lemma syl}
For any $1\leqslant i,j\leqslant m+n$, we have
\begin{equation*}
\psi_p(l^{\pm}_{ij}(z))=\left|\begin{array}{cccc}
l^{\pm}_{11}(z)&\cdots&l^{\pm}_{1p}(z)&l^{\pm}_{1,p+j}(z)
\\
\vdots&&\vdots&\vdots
\\
l^{\pm}_{p1}(z)&\cdots&l^{\pm}_{pp}(z)&l^{\pm}_{p,p+j}(z)
\\
l^{\pm}_{p+i,1}(z)&\cdots&l^{\pm}_{p+i,p}(z)&\framebox{$l^{\pm}_{p+i,p+j}(z)$}
\end{array}\right|.
\end{equation*}
\end{lemm}
As an immediate consequence, we have the following lemma.
\begin{lemm}\label{corpsi}
For $p, i\geqslant 1$, we have
\begin{gather*}
    \psi_{p}:~~k_i^{\pm}(z)\mapsto k_{p+i}^{\pm}(z),\quad e_{i}^{\pm}(z)\mapsto e_{p+i}^{\pm}(z),\quad f_{i}^{\pm}(z)\mapsto f_{p+i}^{\pm}(z).
\end{gather*}
\end{lemm}

\subsection{Surjective homomorphism between two presentations}
We consider $N=2$ first. 
For the case $(m,n)=(2,0)$, the relations between Gaussian generators are the same as in the case $N=2$ in \cite{DF93}. By the superalgebra isomorphism $\zeta_{2|0}$ defined in Proposition \ref{m to n}, we can get the relations between Gaussian generators for the case $(m,n)=(0,2)$. Using equations \eqref{RLL2.1}--\eqref{RLL2.7}, and by calculations similar to those for the non-super case \cite{DF93}, we also can deduce the relations between Gaussian generators for the case $m=1,n=1$ (See the appendix \ref{appe:mn11} for details).
\begin{align}\label{RLL2.8}
 k_i^{ \pm}(z) k_j^{ \pm}(w)&=k_j^{ \pm}(w) k_i^{ \pm}(z), \quad i,j=1,2,
\\
k_1^{+}(z) k_1^{-}(w)&=k_1^{-}(w) k_1^{+}(z),
\\ \label{RLL2.10}
\frac{w_{-} q-q^{-1} z_{+}}{z_{+} q-w_{-} q^{-1}} k_2^{+}(z) k_2^{-}(w)
&=\frac{w_{+} q-q^{-1} z_{-}}{z_{-} q-w_{+} q^{-1}} k_2^{-}(w) k_2^{+}(z),\\ \label{RLL2.12}
\frac{z_{ \pm}-w_{\mp}}{z_{ \pm} q-w_{\mp} q^{-1}} k_2^{\mp}(w)^{-1} k_1^{ \pm}(z)
&=k_1^{ \pm}(z) k_2^{\mp}(w)^{-1} \frac{z_{\mp}-w_{ \pm}}{z_{\mp} q-w_{ \pm}q^{-1}},\\
k_{1}^{\pm}(z)^{-1}x_{1}^{-}(w)k_{1}^{\pm}(z)&=\frac{z_{\mp}q-wq^{-1}}{z_{\mp}-w}x_{1}^{-}(w),\\
k_{1}^{\pm}(z)x_{1}^{+}(w)k_{1}^{\pm}(z)^{-1}&=\frac{z_{\pm}q-wq^{-1}}{z_{\pm}-w}x_{1}^{+}(w),\\ \label{X1}
x_{1}^{\pm}(z)x_{1}^{\pm}(w)&+x_{1}^{\pm}(w)x_{1}^{\pm}(z)=0,\\
k_{2}^{\pm}(z)^{-1}x_{1}^{-}(w)k_{2}^{\pm}(z)&=\frac{z_{\mp}q-wq^{-1}}{z_{\mp}-w}x_{1}^{-}(w),\\
k_{2}^{\pm}(z)x_{1}^{+}(w)k_{2}^{\pm}(z)^{-1}&=\frac{z_{\pm}q-wq^{-1}}{z_{\pm}-w}x_{1}^{+}(w),
\end{align}
\begin{equation}
\begin{aligned}
{[x_{1}^{+}(z),x_{1}^{-}(w)]}&=(q-q^{-1})\{\delta(z^{-1}wq^{c})k_2^{+}(w_{+})
k_1^{+}(w_{+})^{-1}\\
&-\delta(z^{-1}wq^{-c})k_2^{-}(z_{+})k_1^{-}(z_{+})^{-1}\}.
\end{aligned}
\end{equation}
Thus, we have shown that the map $\varphi$ is a homomorphism when $N=2$. The surjectivity of $\varphi$ is straightforward. These results constitute the initial step towards the computations for general $N$.

We begin the proof that $\varphi$ is a surjective homomorphism for $N=3$. By Proposition \ref{m to n} and Lemma \ref{corpsi}, the relations among $k_2^{ \pm}(z), k_3^{ \pm}(z), e_2^{ \pm}(z), f_2^{ \pm}(z)$ can be established, and they coincide with those in $N=2$ case. It remains to derive the relations between the sets $\{k_1^{ \pm}(z)$, $e_1^{ \pm}(z)$, $f_1^{ \pm}(z)\}$ and $\{k_3^{ \pm}(z), e_2^{ \pm}(z)$, $f_2^{ \pm}(z)\}$. To this end, we write $L^{ \pm}(z)$ and $L^{ \pm}(z)^{-1}$ in the following forms:
$$
L^{ \pm}(z)=\left(\begin{array}{ccc}
k_1^{ \pm}(z) & k_1^{ \pm}(z) f_1^{ \pm}(z) & k_{1}^{\pm}(z)f_{1,3}^{\pm}(z)
\\
e_1^{ \pm}(z) k_1^{ \pm}(z) & * & *
\\
e_{3,1}^{ \pm}(z) k_1^{ \pm}(z) & * & *
\end{array}\right),
$$
Let
\begin{align*}
&x^{\pm}=k_{3}^{\pm}(w)^{-1}(-e_{3,1}(w)+e_{2}^{\pm}(w)e_{1}^{\pm}(w)),\\
&y^{\pm}=(-f_{1,3}^{\pm}(w)+f_{1}^{\pm}(w)f_{2}^{\pm}(w))k_{3}^{\pm}(w)^{-1},
\end{align*}
then
$$
L^{ \pm}(w)^{-1}=\left(\begin{array}{ccc}
* & * & y^{\pm}
\\
* & * & -f_2^{ \pm}(w) k_3^{ \pm}(w)^{-1}
\\
x^{\pm} & -k_3^{ \pm}(w)^{-1} e_2^{ \pm}(w) & k_3^{ \pm}(w)^{-1}
\end{array}\right).
$$
where $*$ represent some elements in the $U_q^{R}(\widehat{\mathfrak{g}})$.

For convenience, we rewrite \eqref{RLL2.5}--\eqref{RLL2.7} explicitly as
\begin{equation*}
\left(L_1^{ \pm}(w)^{-1}\right)_{j_1 i_2}^{i_1 i_2} R_{21}\left(zw^{-1}\right)_{k_1 j_2}^{j_1 i_2}\theta_{k_1j_2}^{k_1j_2} L_2^{ \pm}(z)_{k_1 k_2}^{k_1 j_2}\theta_{k_1k_2}^{k_1k_2}=\theta_{i_1i_2}^{i_1i_2}L_2^{ \pm}(z)_{i_1 j_2}^{i_1 i_2}\theta_{i_1j_2}^{i_1j_2} R_{21}\left(zw^{-1}\right)_{j_1 k_2}^{i_1 j_2}\left(L_1^{ \pm}(w)^{-1}\right)_{k_1 k_2}^{j_1 k_2},
\end{equation*}
where $i_1,i_2,k_1,k_2$ are free indices, summations over $j_1,j_2$ are assumed. By taking special values of $i_1,i_2,k_1,k_2$, we obtain the following relations:
\begin{gather}\label{k1k3 rel1}
k_1^{ \pm}(z) k_3^{ \pm}(w)=k_3^{ \pm}(w) k_1^{ \pm}(z),
\\ \label{k1k3 rel2}
\frac{z_{ \pm}-w_{\mp}}{z_{ \pm} q-w_{\mp} q^{-1}} k_3^{\mp}(w)^{-1} k_1^{ \pm}(z)=k_1^{ \pm}(z) k_3^{\mp}(w)^{-1} \frac{z_{\mp}-w_{ \pm}}{z_{\mp} q-w_{ \pm} q^{-1}},
\\ \label{k3e1 rel1}
e_1^{\epsilon}(z) k_3^{\epsilon'}(w)=k_3^{\epsilon'}(w) e_1^{ \epsilon}(z),
\\
k_3^{\epsilon}(w) f_1^{\epsilon'}(z)=f_1^{\epsilon'}(z) k_3^{ \epsilon}(w),
\\
k_1^{\epsilon}(w) f_2^{\epsilon'}(z)=f_2^{\epsilon'}(z) k_1^{ \epsilon}(w),
\\
e_2^{ \epsilon}(w) k_1^{ \epsilon'}(z)=k_1^{ \epsilon'}(z) e_2^{ \epsilon}(w),
\\
e_2^{ \epsilon}(w) f_1^{ \epsilon'}(z)=f_1^{ \epsilon'}(z) e_2^{ \epsilon}(w),
\\
f_2^{ \epsilon}(w) e_1^{ \epsilon'}(z)=e_1^{ \epsilon'}(z) f_2^{ \epsilon}(w),
\end{gather}
where $\epsilon, \epsilon'\in\{+, -\}$. In addition, let $i_1=3, i_2=2, k_1=2, k_2=1$, we have
\begin{align}\notag
&-(-1)^{\delta_{m\leq 1}}\frac{(z-w)}{z q-wq^{-1}}e_1^{ \pm}(z)k_1^{ \pm}(z) k_3^{ \pm}(w)^{-1} e_2^{\pm}(w)
\\  \label{X1X2 pre rel1}
&=\frac{w\left(q-q^{-1}\right)}{zq-wq^{-1}}k_3^{ \pm}(w)^{-1}\left[-e_{3,1}^{\pm}(w)
+e_2^{\pm}(w) e_1^{ \pm}(w)\right]k_1^{ \pm}(z)
\\ \notag
& \quad-d\left(zw^{-1}\right) k_3^{ \pm}(w)^{-1} e_2^{ \pm}(w) e_1^{ \pm}(z) k_1^{ \pm}(z)+\frac{z\left(q-q^{-1}\right)}{z q-w q^{-1}} k_3^{ \pm}(w)^{-1} e_{3,1}^{ \pm}(z) k_1^{ \pm}(z),
\end{align}
and
\begin{align}\notag
& -(-1)^{\delta_{m\leq 1}}\frac{(z_{\mp}-w_{ \pm})}{z_{\mp} q-w_{ \pm} q^{-1}} e_1^{ \pm}(z) k_1^{ \pm}(z) k_3^{\mp}(w)^{-1} e_2^{\mp}(w)
\\ \label{X1X2 pre rel2}
&=\frac{w_{\mp}\left(q-q^{-1}\right)}{z_{ \pm}q-w_{\mp}q^{-1}} k_3^{\mp}(w)^{-1}
\left[-e_{3,1}^{\mp}(w)+e_2^{\mp}(w) e_1^{\mp}(w)\right] k_1^{ \pm}(z)
\\ \notag
& \quad-d\left(z_{ \pm}w_{\mp}^{-1}\right) k_3^{\mp}(w)^{-1} e_2^{\mp}(w) e_1^{ \pm}(z) k_1^{ \pm}(z)+\frac{z_{ \pm}\left(q-q^{-1}\right)}{z_{ \pm} q-w_{\mp} q^{-1}} k_3^{\mp}(w)^{-1} e_{3,1}^{ \pm}(z) k_1^{ \pm}(z).
\end{align}
Here, $d(zw^{-1})=1$ for the case $m=2$ or $m=3$, $d(zw^{-1})=\frac{wq-zq^{-1}}{z q-w q^{-1}}$ for the case $m=1$ or $m=0$.

By calculations similar to the Appendix in \cite{FHS97}, we get that
\begin{equation*}
(z-w)x_1^{+}(z) x_2^{+}(w)=(-1)^{\delta_{m\leq 1}}d\left(zw^{-1}\right)(zq-w q^{-1})x_2^{+}(w) x_1^{+}(z).
\end{equation*}
Similarly, we can show that $x_{1}^{-}(z)$ and $x_{2}^{-}(w)$ satisfy the following commutation relations.
\begin{equation*}
(-1)^{\delta_{m\leq 1}}d\left(zw^{-1}\right)x_1^{-}(z) x_2^{-}(w)=(z-w)x_2^{-}(w) x_1^{-}(z).
\end{equation*}
That is,
\begin{gather}\label{X1X2 rel1}
(z-w)x_1^{+}(z) x_2^{+}(w)=\left(z q-w q^{-1}\right) x_2^{+}(w) x_1^{+}(z), \quad m=2, 3
\\ \label{X1X2 rel2}
(z-w)x_1^{+}(z) x_2^{+}(w)=\left(z q^{-1}-w q\right) x_2^{+}(w) x_1^{+}(z), \quad m=1, 0,\\ \label{X1X2 rel3}
\left(z q-w q^{-1}\right) x_1^{-}(z) x_2^{-}(w)=(z-w) x_2^{-}(w) x_1^{-}(z), \quad m=2, 3,
\\ \label{X1X2 rel4}
\left(z q^{-1}-w q\right) x_1^{-}(z) x_2^{-}(w)=(z-w) x_2^{-}(w) x_1^{-}(z), \quad m=1, 0.
\end{gather}

Since there are twice as many group-like generators, controlling the additional terms in the quantum $R$-matrix requires more effort. Consequently, more Serre relations arise from the $RLL$ relations for $N=3$. In the case $(m,n)=(3,0)$, the serre relations are the same as in the case $N=3$ in \cite{DF93}. By the superalgebra isomorphism $\zeta_{3|0}$ defined in Proposition \ref{m to n}, we can immediately get the serre relations in the case $(m,n)=(0,3)$. Using equations \eqref{X1} and \eqref{X1X2 rel1}--\eqref{X1X2 rel4}, we arrive at the serre relations in the cases $(m,n)=(2,1)$ and $(m,n)=(1,2)$.

Case $m=2$ :
\begin{equation}
\begin{aligned}\label{ser1 m=2}
&\left\{x_1^{\pm}\left(z_1\right)x_1^{\pm}\left(z_2\right)x_2^{\pm}(w)-\left(q+q^{-1}\right)
x_1^{\pm}\left(z_1\right) x_2^{\pm}(w) x_1^{\pm}\left(z_2\right)\right.
\\
& \left.\qquad+\,x_2^{\pm}(w) x_1^{\pm}\left(z_1\right) x_1^{\pm}\left(z_2\right)\right\}+\left\{z_1 \leftrightarrow z_2\right\}=0,
\end{aligned}
\end{equation}

Case $m=1$ :
\begin{equation}
\begin{aligned}\label{ser2 m=1}
&\left\{x_2^{\pm}\left(z_1\right)x_2^{\pm}\left(z_2\right)x_1^{\pm}(w)-\left(q+q^{-1}\right)
x_2^{\pm}\left(z_1\right) x_1^{\pm}(w) x_2^{\pm}\left(z_2\right)\right.
\\
&\left.\qquad+\,x_1^{\pm}(w) x_2^{\pm}\left(z_1\right) x_2^{\pm}\left(z_2\right)\right\}+\left\{z_1 \leftrightarrow z_2\right\}=0,
\end{aligned}
\end{equation}
The verifications of \eqref{ser1 m=2} and \eqref{ser2 m=1} are similar to the case $N=3$ in \cite{DF93}. See the appendix \ref{serre 12} for more details.

This completes the proof that $\varphi$ is a homomorphism for $N=3$. As seen in the proof, The surjectivity of $\varphi$ follows from the fact that $f_{1,3}^{\pm}(z)$ and $e_{3,1}^{\pm}(z)$ are generated by $k_{1}^{\pm}(z)$, $k_{2}^{\pm}(z)$, $k_{3}^{\pm}(z)$, $f_{1}^{\pm}(z)$, $e_{1}^{\pm}(z)$, $f_{2}^{\pm}(z)$, and $e_{2}^{\pm}(z)$.

Now we proceed to the proof for general $N$. Arguing by induction, we assume all relations for the $N-1$ cases are established. As in the $N=3$ case, Proposition \ref{m to n} and Lemma \ref{corpsi} imply that it is sufficient to verify the relations between the generators $k_1^{ \pm}(z), e_1^{ \pm}(z), f_1^{ \pm}(z)$ and $k_N^{ \pm}(z)$, $e_{N-1}^{ \pm}(z), f_{N-1}^{ \pm}(z)$. Using relations \eqref{RLL2.5}--\eqref{RLL2.7}, we have
\begin{align}
k_1^{ \pm}(z) k_N^{ \pm}(w)&=k_N^{ \pm}(w) k_1^{ \pm}(z),
\\
\frac{z_{\pm}-w_{\mp}}{z_{\pm}q-w_{\mp}q^{-1}} k_N^{\mp}(w)^{-1} k_1^{\pm}(z)
&=\frac{z_{\mp}-w_{\pm}}{z_{\mp} q-w_{ \pm} q^{-1}} k_1^{ \pm}(z)k_N^{\mp}(w)^{-1},
\\
k_1^{\epsilon}(z) e_{N-1}^{\epsilon'}(w)&=e_{N-1}^{\epsilon'}(w) k_1^{\epsilon}(z),
\\
k_1^{\epsilon}(z) f_{N-1}^{\epsilon'}(w)&=f_{N-1}^{\epsilon'}(w) k_1^{\epsilon}(z),
\\
k_N^{\epsilon}(z) e_1^{\epsilon'}(w)&=e_1^{\epsilon'}(w) k_N^{\epsilon}(z),
\\
k_N^{\epsilon}(z) f_1^{\epsilon'}(w)&=f_1^{\epsilon}(w) k_N^{ \pm}(z),
\\
f_1^{\epsilon}(z) e_{N-1}^{\epsilon'}(w)&=e_{N-1}^{\epsilon'}(w) f_1^{\epsilon}(z),
\\
e_1^{\epsilon}(z) f_{N-1}^{\epsilon'}(w)&=f_{N-1}^{\epsilon'}(w) e_1^{\epsilon}(z),
\\ \label{fN}
f_1^{\epsilon}(z) f_{N-1}^{\epsilon'}(w)&=f_{N-1}^{\epsilon'}(w) f_1^{\pm}(z),
\\ \label{eN}
e_1^{\epsilon}(z) e_{N-1}^{\epsilon'}(w)&=e_{N-1}^{\epsilon'}(w) e_1^{\epsilon}(z).
\end{align}
Here $\epsilon, \epsilon'\in\{+,-\}$. From \eqref{fN} and \eqref{eN}, we deduce that
\begin{equation}\label{xixj}
x_{i}^{\epsilon}(z)x_{j}^{\epsilon'}(w)=x_{j}^{\epsilon'}(w)x_{i}^{\epsilon}(z),\quad~ \mid i-j\mid>1.
\end{equation}

Otherwise, Using Proposition \ref{m to n} and Lemma \ref{corpsi} to the relations \eqref{X1} and \eqref{X1X2 rel1}--\eqref{X1X2 rel4}, we can get
\begin{align}\label{Xm}
&\left[x_m^{ \pm}(z), x_m^{ \pm}(w)\right]=0,
\\ \label{Xm+}
\left(z-w\right)x_i^{+}(z) x_{i+1}^{+}(w)&=\left(z q_{i+1}-w q_{i+1}^{-1}\right) x_{i+1}^{+}(w) x_i^{+}(z), \quad i=m-1,m,
\\ \label{Xm-}
\left(z q_{i+1}-w q_{i+1}^{-1}\right) x_i^{-}(z)& x_{i+1}^{-}(w)=(z-w) x_{i+1}^{-}(w) x_i^{-}(z), \quad i= m-1,m.
\end{align}
Next, we will verify the serre relation \eqref{serre 2}, that is
\begin{equation}
\begin{aligned}\label{serre +}
&x_{m}^{+}(z_1)x_{m+1}^{+}(z)x_{m}^{+}(z_2)x_{m-1}^{+}(w)-(q+q^{-1})x_{m}^{+}(z_1)
x_{m+1}^{+}(z)x_{m-1}^{+}(w)x_{m}^{+}(z_2)\\
+~&x_{m}^{+}(z_1)
x_{m-1}^{+}(w)x_{m}^{+}(z_2)x_{m+1}^{+}(z)
+x_{m+1}^{+}(z)x_{m}^{+}(z_2)x_{m-1}^{+}(w)x_{m}^{+}(z_1)\\
+~&x_{m-1}^{+}(w)
x_{m}^{+}(z_2)x_{m+1}^{+}(z)x_{m}^{+}(z_1)
+\{z_1\leftrightarrow z_2\}=0.
\end{aligned}
\end{equation}
and
\begin{equation}
\begin{aligned}\label{serre -}
&x_{m}^{-}(z_1)x_{m+1}^{-}(z)x_{m}^{-}(z_2)x_{m-1}^{-}(w)-(q+q^{-1})x_{m}^{-}(z_1)
x_{m+1}^{-}(z)x_{m-1}^{-}(w)x_{m}^{-}(z_2)\\
+~&x_{m}^{-}(z_1)
x_{m-1}^{-}(w)x_{m}^{-}(z_2)x_{m+1}^{-}(z)
+x_{m+1}^{-}(z)x_{m}^{-}(z_2)x_{m-1}^{-}(w)x_{m}^{-}(z_1)\\
+~&x_{m-1}^{-}(w)
x_{m}^{-}(z_2)x_{m+1}^{-}(z)x_{m}^{-}(z_1)
+\{z_1\leftrightarrow z_2\}=0.
\end{aligned}
\end{equation}

We now consider \eqref{serre +}. By \eqref{xixj}--\eqref{Xm-}, we have
\begin{align*}
&x_{m}^{+}(z_1)x_{m+1}^{+}(z)x_{m-1}^{+}(w)x_{m}^{+}(z_2)=\frac{wq-z_{2}q^{-1}}
{w-z_{2}}x_{m}^{+}(z_1)x_{m+1}^{+}(z)x_{m}^{+}(z_2)x_{m-1}^{+}(w),\\
&x_{m}^{+}(z_1)x_{m-1}^{+}(w)x_{m}^{+}(z_2)x_{m+1}^{+}(z)=\frac{wq-z_{2}q^{-1}}
{w-z_{2}}\frac{z_2q^{-1}-zq}{z_2-z}x_{m}^{+}(z_1)x_{m+1}^{+}(z)x_{m}^{+}(z_2)
x_{m-1}^{+}(w),\\
&x_{m+1}^{+}(z)x_{m}^{+}(z_1)x_{m-1}^{+}(w)x_{m}^{+}(z_2)=\frac{z_1-z}
{z_1q^{-1}-zq}\frac{wq-z_{2}q^{-1}}{w-z_{2}}x_{m}^{+}(z_1)x_{m+1}^{+}(z)x_{m}^{+}(z_2)
x_{m-1}^{+}(w),\\
&x_{m-1}^{+}(w)x_{m}^{+}(z_1)x_{m+1}^{+}(z)x_{m}^{+}(z_2)=\frac{wq-z_{1}q^{-1}}
{w-z_{1}}\frac{wq-z_{2}q^{-1}}{w-z_{2}}x_{m}^{+}(z_1)x_{m+1}^{+}(z)x_{m}^{+}(z_2)
x_{m-1}^{+}(w),\\
&x_{m}^{+}(z_2)x_{m+1}^{+}(z)x_{m}^{+}(z_1)x_{m-1}^{+}(w)=-\frac{z_2q^{-1}-zq}
{z_2-z}\frac{z_1-z}{z_1q^{-1}-zq}x_{m}^{+}(z_1)x_{m+1}^{+}(z)x_{m}^{+}(z_2)
x_{m-1}^{+}(w),\\
&x_{m}^{+}(z_2)x_{m-1}^{+}(w)x_{m}^{+}(z_1)x_{m+1}^{+}(z)=-\frac{z_2q^{-1}-zq}
{z_2-z}\frac{wq-z_{1}q^{-1}}{w-z_{1}}x_{m}^{+}(z_1)x_{m+1}^{+}(z)x_{m}^{+}(z_2)
x_{m-1}^{+}(w),\\
&x_{m+1}^{+}(z)x_{m}^{+}(z_2)x_{m-1}^{+}(w)x_{m}^{+}(z_1)=-\frac{wq-z_{1}q^{-1}}
{w-z_{1}}\frac{z_1-z}{z_1q^{-1}-zq}x_{m}^{+}(z_1)x_{m+1}^{+}(z)x_{m}^{+}(z_2)
x_{m-1}^{+}(w),\\
&x_{m}^{+}(z_2)x_{m+1}^{+}(z)x_{m-1}^{+}(w)x_{m}^{+}(z_1)\\
=&-\frac{wq-z_{1}q^{-1}}{w-z_{1}}\frac{z_1-z}{z_1q^{-1}-zq}\frac{z_2q^{-1}-zq}{z_2-z}
x_{m}^{+}(z_1)x_{m+1}^{+}(z)x_{m}^{+}(z_2)x_{m-1}^{+}(w),\\
&x_{m-1}^{+}(w)x_{m}^{+}(z_2)x_{m+1}^{+}(z)x_{m}^{+}(z_1)\\
=&-\frac{z_1-z}{z_1q^{-1}-zq}\frac{wq-z_{1}q^{-1}}{w-z_{1}}\frac{z_2q^{-1}-zq}{z_2-z}
\frac{wq-z_{2}q^{-1}}{w-z_{2}}x_{m}^{+}(z_1)x_{m+1}^{+}(z)x_{m}^{+}(z_2)
x_{m-1}^{+}(w).
\end{align*}
Hence, each summand in relation \eqref{serre +} can be replaced by $x_{m}^{+}(z_1)x_{m+1}^{+}(z)x_{m}^{+}(z_2)x_{m-1}^{+}(w)$. A direct computation then shows that their coefficients sum to zero. Therefore, the relation \eqref{serre +} is true. Similarly, we can verify that the relation \eqref{serre -} also holds.

Therefore we proved that $\varphi$ is a homomorphism. To establish the surjectivity of $\varphi$, it suffices to show that $e_{N,1}^{\pm}(z)$ and $f_{1,N}^{\pm}(z)$ are generated by $k_{i}^{\pm}(z)$, $e_{i}^{\pm}(z)$ and $f_{i}^{\pm}(z)$, since by the induction hypothesis all other elements $e_{i,j}^{\pm}(z)$ and $f_{i,j}^{\pm}(z)$ are already generated by these \hongda{generators}. From \eqref{RLL2.5} and \eqref{RLL2.6}, we derive the relations between $e_{N-1,1}^{\pm}(z)$ and $e_{N,N-1}^{\pm}(z)$, and the relations between $f_{1,N-1}^{\pm}(z)$ and $f_{N-1,N}^{\pm}(z)$, which also involve $e_{N,1}^{\pm}(z)$ and $f_{1,N}^{\pm}(z)$. These relations are analogous to \eqref{X1X2 pre rel1} and \eqref{X1X2 pre rel2} in the $N=3$ case. They imply that $f_{1,N}^{\pm}(z)$ and $e_{N,1}^{\pm}(z)$ are generated by $k_{i}^{\pm}(z)$, $e_{i}^{\pm}(z)$ and $f_{i}^{\pm}(z)$, Consequently, the algebra $U_q^{R}(\widehat{\mathfrak{g}})$ is generated by $k_{i}^{\pm}(z)$, $e_{i}^{\pm}(z)$ and $f_{i}^{\pm}(z)$. So $\varphi$ is surjective.

\subsection{Injectivity of $\varphi$}

Let $\mathbb{A}$ be the localization of $\mathbb{C}[q]$ at the ideal $(q-1)$. To be more exact,
\begin{gather*}
\mathbb{A}=\Bigg\{\,\frac{f(q)}{g(q)}\,\Bigg|\,f(q),g(q)\in \mathbb{C}[q],\ g(1)\neq 0\,\Bigg\}.
\end{gather*}
\hongda{Before we demonstrate that the surjective homomorphism $\varphi$ is an injection, we need to review the $\mathbb{A}$-form for $U_q^{R}(\mathcal{L}\mathfrak{gl}(m|n))$ and $U_q^{D}(\mathcal{L}\mathfrak{gl}(m|n))$}.

\subsubsection{$\mathbb{A}$-form for $U_q^{R}(\mathcal{L}\mathfrak{gl}(m|n))$}
\begin{defi}
The $\mathbb{A}$-form of ~$U_q^{R}(\mathcal{L}\mathfrak{gl}(m|n))$, denote ~$U_{\mathbb{A}}^{R}$, is an $\mathbb{A}$-sub-superalgebra with the generators ${\gamma_{ij}^{+}}^{(r)},{\gamma_{ij}^{-}}^{(r)}$ given by
\begin{equation*}
{\gamma_{ij}^{+}}^{(r)}=\begin{cases}
\displaystyle\frac{{l_{ii}^{+}}^{(0)}-1}{q-1}, &\textrm{if}~~i=j,\ r=0,\\
& \\
\displaystyle\frac{{l_{ij}^{+}}^{(r)}}{q-q^{-1}}, &\textrm{otherwise},
\end{cases}\quad
{\gamma_{ij}^{-}}^{(r)}=\begin{cases}
\displaystyle\frac{{l_{ii}^{-}}^{(0)}-1}{q-1}, &\textrm{if}~~i=j,\ r=0,\\
& \\
\displaystyle\frac{{l_{ij}^{-}}^{(r)}}{q-q^{-1}}, &\textrm{otherwise}.
\end{cases}
\end{equation*}
\end{defi}

Note that all elements ${l_{ij}^{+}}^{(r)}$, ${l_{ij}^{-}}^{(r)}\in U_{\mathbb{A}}^{R}$. The diagonal entries in $U_{\mathbb{A}}^{R}$ satisfy
\begin{gather}\label{modulo:1}
{\gamma_{ii}^{+}}^{(0)}+{\gamma_{ii}^{-}}^{(0)}=-(q-1){\gamma_{ii}^{+}}^{(0)}
{\gamma_{ii}^{-}}^{(0)}\in (q-1)U_{\mathbb{A}}^{R},
\end{gather}
and
\begin{gather}\label{modulo:2}
{l_{ii}^{+}}^{(0)}-1=(q-1){\gamma_{ii}^{+}}^{(0)},\quad {l_{ii}^{-}}^{(0)}-1=(q-1){\gamma_{ii}^{-}}^{(0)}\in (q-1)U_{\mathbb{A}}^{R}.
\end{gather}

The defining relations for $\mathrm{U}_{\mathbb{A}}^{R}$ can be derived from the relations for $U_q^{R}(\mathcal{L}\mathfrak{gl}(m|n))$. More precisely, the superalgebra $U_{\mathbb{A}}^{R}$ admits the defining relations
\begin{align}\label{UA:1}
&{\gamma_{ii}^{+}}^{(0)}+{\gamma_{ii}^{-}}^{(0)}+(q-1){\gamma_{ii}^{+}}^{(0)}
{\gamma_{ii}^{-}}^{(0)}=0,\ \text{for }i\in I,
\\ \label{UA:2}
&{\gamma_{ji}^{+}}^{(0)}={\gamma_{ij}^{-}}^{(0)}=0,\ \textrm{for}~~1\leqslant i<j\leqslant m+n,
\end{align}
together with the substitutions that replace ${l_{ij}^{+}}^{(r)}$ and ${l_{ij}^{-}}^{(r)}$ by $\left(q-q^{-1}\right){\gamma_{ij}^{+}}^{(r)}$ and $\left(q-q^{-1}\right){\gamma_{ij}^{-}}^{(r)}$, respectively, for all $i,j\in I$, $r\geqslant 0$, except that for $i=j,r=0$, the elements ${l_{ii}^{+}}^{(0)}$ and ${l_{ii}^{-}}^{(0)}$ are replaced by  $(q-1){\gamma_{ii}^{+}}^{(0)}+1$ and $(q-1){\gamma_{ii}^{-}}^{(0)}+1$.

Similar to Proposition \ref{base}, the algebra $U_{\mathbb{A}}^{R}$ is spanned by the ordered monomials \eqref{base:Uq} in generators ${\gamma_{ij}^{+}}^{(r)}$ and ${\gamma_{ij}^{-}}^{(r)}$,  subject to the conditions \eqref{base:1} and \eqref{base:2} as an $\mathbb{A}$-linear superspace.

Let $\overline{U_{\mathbb{A}}^{R}}=U_{\mathbb{A}}^{R}/(q-1)U_{\mathbb{A}}^{R}$. If we denote the image of ${\gamma_{ij}^{\pm}}^{(r)}$ in the quotient again by ${\gamma_{ij}^{\pm}}^{(r)}$, then these elements generate the algebra $\overline{U_{\mathbb{A}}^{R}}$.
\begin{lemm}
In the algebra $\overline{U_{\mathbb{A}}^{R}}$, we have
\begin{equation*}
{\gamma^{\pm}_{ii}}^{(0)}= \frac{k_{i,0}^{\pm}-1}{q-1},\quad
{\gamma^{\pm}_{ii}}^{(r)}= \frac{k_{i,\pm r}^{\pm}}{q-q^{-1}},\quad r>0;
\end{equation*}
\begin{equation*}
{\gamma^{+}_{ij}}^{(0)}=\frac{{f_{ij}^{+}}^{(0)}}{q-q^{-1}}, \quad{\gamma^{-}_{ij}}^{(0)} = 0,\quad{\gamma^{\pm}_{ij}}^{(r)}= \frac{{f_{ij}^{\pm}}^{(r)}}{q-q^{-1}},\quad r>0 ,\quad i<j;
\end{equation*}
\begin{equation*}
{\gamma^{+}_{ji}}^{(0)}=0,\quad {\gamma^{-}_{ji}}^{(0)}= \frac{{e_{ji}^{-}}^{(0)}}{q-q^{-1}},\quad {\gamma^{\pm}_{ji}}^{(r)}= \frac{{e_{ji}^{\pm}}^{(r)}}{q-q^{-1}}, \quad r>0 ,\quad i<j ;
\end{equation*}
\end{lemm}
\begin{proof}
According to \eqref{modulo:2} and \eqref{Gauss ii}--\eqref{Gauss ji}, the above conclusion follows from a direct calculation.
\end{proof}

\begin{lemm}(\cite[Theorem 2.14]{LWZ24})\label{gamma}
There exists an isomorphism of superalgebras between $U(\mathcal{L}\mathfrak{gl}(m|n))$ to $\overline{U_{\mathbb{A}}^{R}}$ such that
\begin{align*}
&E_{ij}^{(0)}\mapsto (-1)^{|i||j|}{\gamma^{+}_{ij}}^{(0)},\quad~E_{ji}^{(0)}\mapsto -(-1)^{|i||j|}{\gamma^{-}_{ji}}^{(0)} \quad~{\rm for~}1 \leqslant i\leqslant j\leqslant m+n,\\
&E_{ij}^{(-r)}\mapsto (-1)^{|i||j|}{\gamma^{+}_{ij}}^{(r)},\quad~E_{ij}^{(r)}\mapsto -(-1)^{|i||j|}{\gamma^{-}_{ij}}^{(r)}\quad~{\rm for~}i,j\in I, r>0.
\end{align*}

\end{lemm}

\maketitle

\subsubsection{$\mathbb{A}$-form for $U_q^{D}(\mathcal{L}\mathfrak{gl}(m|n))$}
Similar to the definition of $\mathbb{A}$-form for $U_q^D(\widehat{\mathfrak{sl}(m|n)})$ in \cite{BCFK22}, we define the $\mathbb{A}$-form for $U_q^{D}(\mathcal{L}\mathfrak{gl}(m|n))$.

\begin{defi}
For $i \in I$ and $u, v, s \in \mathbb{Z}$ with $s>0$, we define
\begin{gather*}
    K_{i;u,v,s}=\prod_{j=1}^s \frac{K_{i, u+v}^{+}-K_{i, u+v}^{-}}{q^j-q^{-j}},\quad K_{i;u,s}=\prod_{j=1}^s \frac{K_{i,0} q^{u-j+1}-K_{i,0}^{-1}q^{j-u-1}}{q^j-q^{-j}}.
\end{gather*}
The $\mathbb{A}$-form of ~$U_q^{D}(\mathcal{L}\mathfrak{gl}(m|n))$, denoted ~$U_{\mathbb{A}}^{D}$, is the $\mathbb{A}$-sub-superalgebra generated by the elements $H_{i, \pm s}$, $K_i^{ \pm 1},K_{i;u,v,s}$, $K_{i;u,s}$, $X_{i, r}^{ \pm}:=\frac{X_{i, r}^{ \pm}}{q-q^{-1}}$ with $r\in \mathbb{Z}$.
\end{defi}

Define the quotient algebra 
$\overline{U_{\mathbb{A}}^{D}}$=$U_{\mathbb{A}}^{D}/(q-1)U_{\mathbb{A}}^{D}$. Then we have the following lemma.
\begin{lemm}
In the algebra $\overline{U_{\mathbb{A}}^{D}}$, we obtain the following explicit expressions:
\begin{equation*}
 K_{i;u,v,s}=\begin{cases}
\frac{1}{s!} \left(\frac{K_{i,0}-1}{q-1}\right)^{s} & if \quad~ u+v=0;\\
\frac{1}{s!}\left(\frac{K_{i, u+v}^{+}}{q-q^{-1}}\right)^{s} & if \quad~ u+v>0;\\
\frac{1}{s!}\left(-\frac{K_{i, u+v}^{-}}{q-q^{-1}}\right)^{s} & if \quad~ u+v<0.
\end{cases}
\end{equation*}
\begin{equation*}
K_{i;u,s}=\frac{1}{s!} \left(\frac{K_{i,0} -1}{q-1}\right)^{s}=(K_{i;u,v,s})_{u+v=0}.
\end{equation*}
\end{lemm}

For each $(\alpha,r)\in\Delta^{+}\times\mathbb{Z}$, let $\alpha=\alpha_{i_1}+\cdots+\alpha_{i_p}$ and fix $r=r_1+\cdots+r_p$ with $r_i\in\mathbb{Z}$. In the algebra $\overline{U_{\mathbb{A}}^{D}}$, we have 
$$
X_{\pm\alpha,r}=\left[\cdots\left[\left[X_{i_{1},r_{1}}^{\pm},X_{i_{2},r_{2}}^{\pm}\right],
X_{i_{3},r_{3}}^{\pm}\right],\cdots,X_{i_{p},r_{p}}^{\pm}\right].
$$
According to Definition \ref{Drinfeld def}, we also have :
\begin{align*}
&K_{i;u,v,s}X_{i,k}^{\pm}=X_{i,k}^{\pm}K_{i;u,v,s},\\
&\left[X_{i,k}^{+},X_{j,l}^{-}\right]=0,\\
&\left[X_{i,k}^{\pm},X_{j,l}^{\pm}\right]=0,\quad~&\alpha_{ij}=0,\\
&\left[X_{i,k+1}^{\pm},X_{j,l}^{\pm}\right]+(-1)^{|\alpha_i||\alpha_j|}\left[X_{j,l}^{\pm},X_{i,k+1}^{\pm}\right]=0,\quad~&\alpha_{ij}\neq0.
\end{align*}
Then, we can rearrange the ordered monomials in Theorem \ref{UD basis} to obtain a basis  for the algebra $\overline{U_{\mathbb{A}}^{D}}$, and the order of this basis is consistent with that of the ordered monomials given in Proposition \ref{base}.

\begin{coro}\label{base UAQ}
The set of monomials
$$
\{X_{\mathrm{\textbf{m}}}^{+}\prod_{u+v=-j}K_{i;u,v,1}^{\mathrm{\textbf{h}}(i,-j)}
\prod_{u+v=0}K_{i;u,v,1}^{\mathrm{\textbf{h}}(i,0)}
\prod_{u+v=j}K_{i;u,v,1}^{\mathrm{\textbf{h}}(i,j)}X_{\mathrm{\textbf{m}^{'}}}^{-}\}
$$
with $\mathrm{\textbf{m}}$, $\mathrm{\textbf{m}^{'}}\in\mathrm{\textbf{M}}$,
$\mathrm{\textbf{h}}\in\mathrm{\textbf{H}}$, $u, v, j\in\mathbb{Z}$, and $j>0$, forms an $\mathbb{A}$-linear basis of the algebra $\overline{U_{\mathbb{A}}^{D}}$.
Here,
$$
X_{\mathrm{\textbf{m}}}^{+}=\prod\limits_{(\alpha,r)\in \Delta^+\times \mathbb{Z}}^{\leftarrow}(X_{\alpha,r})^{\mathrm{\textbf{m}}(\alpha,r)},\qquad
X_{\mathrm{\textbf{m}^{'}}}^{-}=\prod\limits_{(\alpha,r)\in \Delta^+\times \mathbb{Z}}^{\rightarrow}(X_{-\alpha,r})^{\mathrm{\textbf{m}^{'}}(\alpha,r)},
$$
\end{coro}

\subsubsection{Proof of the injectivity of $\varphi$}

\begin{lemm}\label{varphi bar}
The superalgebraic surjective homomorphism $\varphi$ induces a superalgebraic homomorphism $\overline{\varphi}:\overline{U_{\mathbb{A}}^{D}}\rightarrow\overline{U_{\mathbb{A}}^{R}}$, with the action on generators defined by:
\begin{align*}
&(K_{i;u,v,1})_{u+v=0}\mapsto {\gamma^{+}_{ii}}^{(0)},\quad~~
-(K_{i;u,v,1})_{u+v=0}\mapsto {\gamma^{-}_{ii}}^{(0)};\\
&(K_{i;u,v,1})_{u+v=r}\mapsto {\gamma^{+}_{ii}}^{(r)},\quad~~
-(K_{i;u,v,1})_{u+v=-r}\mapsto {\gamma^{-}_{ii}}^{(r)};\\
-&X_{i,0}^{+}\mapsto {\gamma^{-}_{i+1,i}}^{(0)},\quad~~
X_{i,r}^{+}\mapsto {\gamma^{+}_{i+1,i}}^{(r)},\quad~~
-X_{i,-r}^{+}\mapsto {\gamma^{-}_{i+1,i}}^{(r)},\\
&X_{i,0}^{-}\mapsto {\gamma^{+}_{i,i+1}}^{(0)},\quad~~
X_{i,r}^{-}\mapsto {\gamma^{+}_{i,i+1}}^{(r)},\quad~~
-X_{i,-r}^{-}\mapsto {\gamma^{-}_{i,i+1}}^{(r)}.
\end{align*}
\end{lemm}

\begin{lemm}\label{induction}
In the algebra $\overline{\mathrm{U}_{\mathbb{A}}^{R}}$, for $i<j$, we have the following relations:
\begin{align}\label{gamma1}
{\gamma_{ij}^{+}}^{(r)}&=(-1)^{|j-1|} \varepsilon_{i,j-1;j-1,j} \left[{\gamma_{ij-1}^{+}}^{(r)},{\gamma_{j-1j}^{+}}^{(0)}\right];\\ \label{gamma2}
{\gamma_{ji}^{-}}^{(r)}&=-(-1)^{|j-1|}\varepsilon_{j,j-1;j-1,i}
\left[{\gamma_{jj-1}^{-}}^{(0)},{\gamma_{j-1i}^{-}}^{(r)}\right];\\ \label{gamma3}
{\gamma_{ji}^{+}}^{(r)}&=(-1)^{|j-1|}\varepsilon_{j,j-1;j-1,i}
\left[{\gamma_{jj-1}^{+}}^{(1)},{\gamma_{j-1i}^{+}}^{(r-1)}\right],~r\geqslant j-i;\\ \label{gamma4}
{\gamma_{ij}^{-}}^{(r)}&=-(-1)^{|j-1|}\varepsilon_{i,j-1;j-1,j}
\left[{\gamma_{ij-1}^{-}}^{(r-1)},{\gamma_{j-1j}^{-}}^{(1)}\right],~r\geqslant j-i.
\end{align}
More explicitly,
\begin{align}\label{gamma5}
{\gamma_{ij}^{+}}^{(r)}&=\varepsilon_{ij}
\left[\cdots\left[\left[{\gamma_{ii+1}^{+}}^{(r)},{\gamma_{i+1i+2}^{+}}^{(0)}\right], {\gamma_{i+2i+3}^{+}}^{(0)}\right],\cdots,{\gamma_{j-1j}^{+}}^{(0)}\right];\\ \label{gamma6}
{\gamma_{ji}^{-}}^{(r)}&=(-1)^{j-i-1}\varepsilon_{ji}
\left[{\gamma_{jj-1}^{-}}^{(0)},\cdots,\left[{\gamma_{i+3i+2}^{-}}^{(0)},
\left[{\gamma_{i+2i+1}^{-}}^{(0)},{\gamma_{i+1i}^{-}}^{(r)}\right]\right]\cdots\right];\\
\label{gamma7}
{\gamma_{ji}^{+}}^{(r)}&=\varepsilon_{ji}
\left[{\gamma_{jj-1}^{+}}^{(1)},\cdots,\left[{\gamma_{i+3i+2}^{+}}^{(1)},
\left[{\gamma_{i+2i+1}^{+}}^{(1)},{\gamma_{i+1i}^{+}}^{(r+1+i-j)}\right]\right]\cdots\right],~r\geqslant j-i;\\ \label{gamma8}
{\gamma_{ij}^{-}}^{(r)}&=(-1)^{j-i-1}\varepsilon_{ij}
\left[\cdots\left[\left[{\gamma_{ii+1}^{-}}^{(r+1+i-j)},{\gamma_{i+1i+2}^{-}}^{(1)}\right], {\gamma_{i+2i+3}^{-}}^{(1)}\right],\cdots,{\gamma_{j-1j}^{-}}^{(1)}\right],~r\geqslant j-i.
\end{align}
Here,
$$
\varepsilon_{ij}=\varepsilon_{i,j-1;j-1,j}\varepsilon_{i,j-2;j-2,j-1}\cdots\varepsilon_{i,i+1;i+1,i+2},\quad~~
\varepsilon_{ji}=\varepsilon_{j,j-1;j-1,i}\varepsilon_{j-1,j-2;j-2,i}\cdots\varepsilon_{i+2i+1;i+1i}.
$$
\end{lemm}
\begin{proof}
According to Lemma \ref{gamma}, we can obtain
\begin{align*}
\left[(-1)^{|i||j-1|}{\gamma_{ij-1}^{+}}^{(r)},(-1)^{|j-1||j|}{\gamma_{j-1j}^{+}}^{(0)} \right]=(-1)^{|i||j|}{\gamma_{ij}^{+}}^{(r)},\\
\left[-(-1)^{|j||j-1|}{\gamma_{jj-1}^{-}}^{(0)},-(-1)^{|j-1||i|}{\gamma_{j-1i}^{-}}^{(r)} \right]=-(-1)^{|j||i|}{\gamma_{ji}^{-}}^{(r)},\\
\left[(-1)^{|j||j-1|}{\gamma_{jj-1}^{+}}^{(1)},(-1)^{|j-1||i|}{\gamma_{j-1i}^{+}}^{(r-1)} \right]=(-1)^{|j||i|}{\gamma_{ji}^{+}}^{(r)},\\
\left[-(-1)^{|i||j-1|}{\gamma_{ij-1}^{-}}^{(r-1)},-(-1)^{|j-1||j|}{\gamma_{j-1j}^{-}}^{(1)} \right]=-(-1)^{|i||j|}{\gamma_{ji}^{-}}^{(r)}.
\end{align*}
Simplifying the above equations yields that the relations \eqref{gamma1}--\eqref{gamma4} hold true. Furthermore, the relations \eqref{gamma5}--\eqref{gamma8} can be recursively derived from \eqref{gamma1}--\eqref{gamma4}. Hence, the statement of the lemma follows.
\end{proof}

\begin{lemm}\label{gammaB}
In the algebra $\overline{\mathrm{U}_{\mathbb{A}}^{R}}$, for $i<j$, we have the following relations:
\begin{equation}\label{gamma9}
\left[{\gamma_{ii}^{-}}^{(r)},{\gamma_{ij}^{+}}^{(s)} \right]=\delta_{r\geq s}(-1)^{|i|}{\gamma_{ij}^{-}}^{(r-s)}-\delta_{r\geq s}(-1)^{|i|}{\gamma_{ij}^{+}}^{(s-r)}.
\end{equation}
Moreover, for $r>j-i$,
\begin{align}\label{gamma10}
{\gamma_{ji}^{+}}^{(1)}&=\varepsilon_{ji}\left[{\gamma_{j,j-1}^{-}}^{(r-2)},\left[\cdots,
\left[{\gamma_{i+2,i+1}^{+}}^{(1)},{\gamma_{i+1,i}^{+}}^{(r+1+i-j)}\right]\cdots\right]\right],
\\ \label{gamma11}
{\gamma_{ij}^{-}}^{(1)}&=\varepsilon_{ij}\left[\cdots\left[{\gamma_{ii+1}^{-}}^{(1)},
{\gamma_{i+1i+2}^{+}}^{(0)}\right],\cdots,{\gamma_{j-1j}^{+}}^{(0)}\right].
\end{align}
\end{lemm}
\begin{proof}
By Lemma \ref{gamma}, we have
\begin{align*}
\left[-(-1)^{|i||i|}{\gamma_{ii}^{-}}^{(r)},(-1)^{|i||j|}{\gamma_{ij}^{+}}^{(s)} \right]=-(-1)^{|i||j|}{\gamma_{ij}^{-}}^{(r-s)},\quad~r>s,\\
\left[-(-1)^{|i||i|}{\gamma_{ii}^{-}}^{(r)},(-1)^{|i||j|}{\gamma_{ij}^{+}}^{(s)} \right]=(-1)^{|i||j|}{\gamma_{ij}^{+}}^{(s-r)},\quad~r\leqslant s,
\end{align*}
that is,
\begin{align*}
\left[{\gamma_{ii}^{-}}^{(r)},{\gamma_{ij}^{+}}^{(s)} \right]=(-1)^{|i|}{\gamma_{ij}^{-}}^{(r-s)},\quad~r>s,\\
\left[{\gamma_{ii}^{-}}^{(r)},{\gamma_{ij}^{+}}^{(s)} \right]=-(-1)^{|i|}{\gamma_{ij}^{+}}^{(s-r)},\quad~r\leqslant s.
\end{align*}
So, the relation \eqref{gamma9} holds.

Furthermore, Using relation \eqref{gamma9} and super-Jacobi identity, we can derive for $r>j-i$:
\begin{align*}
{\gamma_{ji}^{+}}^{(1)}&=(-1)^{\mid j\mid+1}\left[{\gamma_{jj}^{-}}^{(r-1)},{\gamma_{ji}^{+}}^{(r)}\right]\\
&=(-1)^{\mid j\mid+1}\varepsilon_{ji}\left[{\gamma_{jj}^{-}}^{(r-1)},\left[{\gamma_{j,j-1}^{+}}^{(1)},\cdots,
\left[{\gamma_{i+2,i+1}^{+}}^{(1)},{\gamma_{i+1,i}^{+}}^{(r+1+i-j)}\right]\cdots\right]\right]\\
&=\varepsilon_{ji}\left[{\gamma_{j,j-1}^{-}}^{(r-2)},\left[\cdots,
\left[{\gamma_{i+2,i+1}^{+}}^{(1)},{\gamma_{i+1,i}^{+}}^{(r+1+i-j)}\right]\cdots\right]\right],\\
{\gamma_{ij}^{-}}^{(1)}&=(-1)^{\mid i\mid}
\left[{\gamma_{ii}^{-}}^{(r+1)},{\gamma_{ij}^{+}}^{(r)}\right]\\
&=(-1)^{\mid i\mid}\varepsilon_{ij}
\left[{\gamma_{ii}^{-}}^{(r+1)},\left[\cdots\left[{\gamma_{ii+1}^{+}}^{(r)},
{\gamma_{i+1i+2}^{+}}^{(0)}\right],\cdots,{\gamma_{j-1j}^{+}}^{(0)}\right]\right]\\
&=\varepsilon_{ij}\left[\cdots\left[{\gamma_{ii+1}^{-}}^{(1)},
{\gamma_{i+1i+2}^{+}}^{(0)}\right],\cdots,{\gamma_{j-1j}^{+}}^{(0)}\right].
\end{align*}
Hence, the relations \eqref{gamma10} and \eqref{gamma11} hold true.
\end{proof}

According to Lemma \ref{induction} and \ref{gammaB}, we can obtain that all elements ${\gamma^{\pm}_{ij}}^{(r)}$, ${\gamma^{\pm}_{ji}}^{(r)}$ are generated by ${\gamma^{\pm}_{i,i+1}}^{(r)}$ and ${\gamma^{\pm}_{i+1,i}}^{(r)}$, which implies that $\overline{\varphi}$ is surjective.

Next, we show that $\overline{\varphi}$ is injective. Similar to Proposition \ref{base}, the ordered monomials \eqref{base:Uq} in generators ${\gamma_{ij}^{+}}^{(r)}$ and ${\gamma_{ij}^{-}}^{(r)}$ form an $\mathbb{A}$-linear basis of the algebra $\overline{U_{\mathbb{A}}^{R}}$. According to Lemma \ref{varphi bar}--\ref{gammaB} and corollary \ref{base UAQ}, the preimage of this basis under the map $\overline{\varphi}$ forms a basis for $\overline{U_{\mathbb{A}}^{D}}$. Consequently, $\overline{\varphi}$ is an injective map.

Finally, the injectivity of $\varphi$ follows directly from its preservation of the $\mathbb{A}$-form. This completes the proof of the main Theorem \ref{main theo}.

\section{Isomorphism between the $R$-matrix and Drinfeld presentations of $U_q(\widehat{\mathfrak{sl}(m|n)})$}

In this section, we will use the quantum Berezinian of $U_{q}(\widehat{\mathfrak{gl}(m|n))}$ as defined by Jing, Li and Zhang \cite{JLZ25} to introduce the $R$-matrix presentation of $U_{q}(\widehat{\mathfrak{sl}(m|n))}$, and prove its isomorphism to the Drinfeld presentation.

The quantum Berezinian of $L^{\pm}(z)$ was defined by Jing, Li and Zhang \cite{JLZ25}. Using \cite[Theorem 4.1]{JLZ25} and the definition of $k_i^{\pm}(z)$, it can be written as the following power series with coefficients in $U_{q}^{R}(\widehat{\mathfrak{gl}(m|n)})$:
\begin{equation*}
\begin{aligned}
B_q(L^{\pm}(z) )
=k_{1}^{\pm}(z)k_{2}^{\pm}(zq^{2})\cdots k_{m}^{\pm}(zq^{2m-2})\times {k_{m+1}^{\pm}(zq^{2m-2})}^{-1}
\cdots {k_{m+n}^{\pm}(zq^{2m-2n})}^{-1}.
\end{aligned}
\end{equation*}

\begin{rema}
The quantum Berezinian $B_q(L^{\pm}(z))$ multiplied by a formal power series \emph{exp}$\sum_{p=1}^{\infty}\mathfrak{f}_{\mp p}z^{\mp p}$ is the quantum Berezinian defined in \cite{JLZ25}, where the coefficients $\mathfrak{f}_{\mp p}$ are rational functions in $q$ and $q^{\frac{c}{2}}$. Therefore, the coefficients of these two quantum Berezinians differ only by a rational function in $q$ and $q^{\frac{c}{2}}$.
\end{rema}

The following theorem provides a family of central elements in $U_{q}^{R}(\widehat{\mathfrak{gl}(m|n)})$.
\begin{theo}(\cite[Corollary 3.5]{JLZ25})\label{center}
The coefficients of the quantum Berezinian belong to the center of ~$U_{q}^{R}(\widehat{\mathfrak{gl}(m|n)})$.
\end{theo}

Let $\mathfrak{C}_{m|n}$ denote the sub-superalgebra of $U_{q}^{R}(\widehat{\mathfrak{gl}(m|n)})$ generated by the coefficients of the quantum Berezinian. Then $\mathfrak{C}_{m|n}$ is a commutative algebra over $\mathbb{C}(q)$.
\begin{defi}
    The superalgebra $U_q^{R}({\widehat{\mathfrak{sl}(m|n)}})$ is defined to be the following sub-superalgebra of $U_{q}^{R}(\widehat{\mathfrak{gl}(m|n)})$: 
    $$U_q^{R}({\widehat{\mathfrak{sl}(m|n)}}):=\Big\{y\in U_{q}^{R}(\widehat{\mathfrak{gl}(m|n)})\Big|\mu_{g^{\pm}}(y)=y {\rm ~for~all~ } g^{\pm}\Big\},$$
    where the map $\mu_{g^{\pm}}$ is the automorphism of $U_{q}^{R}(\widehat{\mathfrak{gl}(m|n)})$ given by
    \begin{equation*}
        \mu_{g^{\pm}}:~~L^{+}(z)\mapsto g^{+}(z)L^{+}(z),\qquad
        L^{-}(z)\mapsto g^{-}(z)L^{-}(z),
    \end{equation*}
where 
$g^{\pm}(z)=g_0^{\pm}+g_1^{\pm}z^{\mp 1}+g_2^{\pm}z^{\mp 2}+\cdots\quad\in\mathbb{C}[[z^{\mp 1}]]$ 
with $g_0^+g_0^-=1$.
\end{defi}

\begin{lemm}\label{mug_action}
For any $1\leqslant i, j\leqslant m+n$ with $i<j$, we have
\begin{align}\label{ki}
\mu_{g^{\pm}}(k_{i}^{\pm}(z))&=g^{\pm}(z)k_{i}^{\pm}(z);\\ 
\mu_{g^{\pm}}(e_{ji}^{\pm}(z))&=e_{ji}^{\pm}(z);\\ \label{fij}
\mu_{g^{\pm}}(f_{ij}^{\pm}(z))&=f_{ij}^{\pm}(z).
\end{align}
\end{lemm}
\begin{proof}
By \eqref{Gauss ii}--\eqref{Gauss ji}, we deduce
$$
\mu_{g^{\pm}}(k_{1}^{\pm}(z))=g^{\pm}(z)k_{1}^{\pm}(z),\quad\mu_{g^{\pm}}(e_{j1}^{\pm}(z))
=e_{j1}^{\pm}(z), \quad\mu_{g^{\pm}}(f_{1j}^{\pm}(z))=f_{1j}^{\pm}(z).
$$
An inductive argument on $i$ and $j$, then establishes the general relations \eqref{ki}--\eqref{fij}.
\end{proof}

\begin{lemm}\label{fact 1}
Let $\mathcal{A}_q$ be a commutative associative algebra over ~$\mathbb{C}(q)$ and
$$
a^{\pm}(z) = 1+ a_{1}^{\pm}z^{\mp 1}+ a_{2}^{\pm}z^{\mp 2}+ \ldots \in \mathcal{A}_q[[z^{\mp}]].
$$
Then, for any positive integer $K$ and $t\in\mathbb{C}(q)$ not a root of unity, there exists a unique formal series
$$
\tilde{a}^{\pm}(z) = 1+ \tilde{a}_{1}^{\pm}z^{\mp 1}+ \tilde{a}_{2}^{\pm}z^{\mp 2}+ \ldots \in \mathcal{A}_q[[z^{\mp}]]
$$
such that
\begin{equation}\label{seriesexpansion}
 a^{\pm}(z) = \tilde{a}^{\pm}(z) \tilde{a}^{\pm}(zt) \cdots \tilde{a}^{\pm}(zt^{K-1}).
 \end{equation}
\end{lemm}
\begin{proof}
Comparing the coefficient of $z^{\mp}$ in \eqref{seriesexpansion}, we find
$$
a_{1}^{\pm}=(1+t^{\mp 1}+\cdots+t^{\mp(K-1)})\tilde{a}_{1}^{\pm},{\rm ~that~is},\quad \tilde{a}_{1}^{\pm}=\frac{a_{1}^{\pm}}{1+t^{\mp1}+\cdots+t^{\mp(K-1)}}.
$$
For $s\geqslant 2$, comparing coefficients of $z^{\mp s}$ yields
$$
a_{s}^{\pm}=(1+t^{\mp s}+\cdots+t^{\mp s(K-1)})\tilde{a}_{s}^{\pm}+p_s(\tilde{a}_{1}^{\pm},\tilde{a}_{2}^{\pm},\cdots,
\tilde{a}_{s-1}^{\pm}),
$$
where $p_s(\tilde{a}_{1}^{\pm},\tilde{a}_{2}^{\pm},\cdots,\tilde{a}_{s-1}^{\pm})$ is a polynomial in $\tilde{a}_{1}^{\pm},\tilde{a}_{2}^{\pm},\cdots,\tilde{a}_{s-1}^{\pm}$. 
By induction, each $\tilde{a}_{s}^{\pm}$ is uniquely expressed as a polynomial in $a_{1}^{\pm},a_{2}^{\pm},\cdots,a_{s-1}^{\pm}$. The constant term is trivially consistent. This completes the proof.
\end{proof}

\begin{prop}
For $m\neq n$, we have
$$
U_q^{R}(\widehat{\mathfrak{gl}(m|n)})\cong \mathfrak{C}_{m|n}\otimes U_q^{R}({\widehat{\mathfrak{sl}(m|n)}}).
$$
\end{prop}
\begin{proof}
The quantum Berezinian can be written in a formal power series
$$
B_q(L^{\pm}(z))=b_0^{\pm}+b_1^{\pm}z^{\pm 1}+b_2^{\pm}z^{\pm 2}+\cdots.
$$
By the definition of $B_q(L^{\pm}(z))$, the constant term
$$
b_0^{\pm}={l_{11}^{\pm}}^{(0)}{l_{22}^{\pm}}^{(0)}\cdots{l_{mm}^{\pm}}^{(0)}\left(
{l_{m+1,m+1}^{\pm}}^{(0)}\right)^{-1}\cdots\left({l_{m+n,m+n}^{\pm}}^{(0)}\right)^{-1}
$$
are invertible. Set $a^{\pm}(z)=\left(b_0^{\pm}\right)^{-1}B_q(L^{\pm}(z))$.

Assume that $m>n$. By lemma \ref{fact 1}, there exists $\tilde{a}^{\pm}(z)\in 1+z^{-1}\mathfrak{C}_{m|n}[[z^{-1}]]$ such that
$$
a^{\pm}(z)=\tilde{a}^{\pm}(z)\tilde{a}^{\pm}(zq^2)\cdots\tilde{a}^{\pm}(zq^{2m-2n-2}).
$$
Since $\mu_{g^{\pm}}\left({l_{ii}^{\pm}}^{(0)}\right)=g_0^{\pm}{l_{ii}^{\pm}}^{(0)}$ for $i=1,2,\cdots,m+n$, the action of $\mu_{g^{\pm}}$ on $a^{\pm}(z)$ is given by
\begin{gather*}
    \mu_{g^{\pm}}\left(a^{\pm}(z)\right)=\left(g_0^{\pm}\right)^{n-m}g^{\pm}(z)g^{\pm}(zq^{2})\cdots g^{\pm}(zq^{2m-2n-2})a^{\pm}(z).
\end{gather*}
The uniqueness of the expression of $\mu_{g^{\pm}}\left(a^{\pm}(z)\right)$ forces
$$
\mu_{g^{\pm}}\left(\tilde{a}^{\pm}(z)\right)=\left(g_0^{\pm}\right)^{-1}g^{\pm}(z)\tilde{a}^{\pm}(z).
$$

Now we define $\widetilde{l_{ij}^{\pm}}(z)=\left(\tilde{a}^{\pm}(z)\right)^{-1}{l_{ii}^{\pm}}^{(0)}l_{ij}^{\pm}(z)$. 
Then $\mu_{f^{\pm}}\left(\widetilde{l_{ij}^{\pm}}(z)\right)=\widetilde{l_{ij}^{\pm}}(z)$, so, $\widetilde{l_{ij}^{\pm}}(z)\in U_q^{R}({\widehat{\mathfrak{sl}(m|n)}})$. This implies that
$$
U_q^{R}({\widehat{\mathfrak{gl}(m|n)}})=\mathfrak{C}_{m|n}\cdot U_q^{R}({\widehat{\mathfrak{sl}(m|n)}}).
$$
Moreover,
$$
\mathfrak{C}_{m|n}\cap U_q^{R}({\widehat{\mathfrak{sl}(m|n)}})=\varnothing,
$$
so
$$
U_q^{R}({\widehat{\mathfrak{gl}(m|n)}})\cong \mathfrak{C}_{m|n}\otimes U_q^{R}({\widehat{\mathfrak{sl}(m|n)}}).
$$

The case $m<n$ is similar.
\end{proof}

The following lemma presents the $R$-matrix presentation of $U_{q}(\widehat{\mathfrak{sl}(m|n))}$.
\begin{lemm}\label{SL generators}
For any $m,n\geqslant 0$, the coefficients of the series
\begin{equation}\label{SL}
{k_{i}^{\pm}(z)}^{-1}k_{i+1}^{\pm}(z),\quad e_{i}^{\pm}(z),\quad f_{i}^{\pm}(z), \quad {\rm for~}1\leqslant i\leqslant m+n-1,
\end{equation}
along with the central element $q^{\frac{c}{2}}$ generate the subalgebra $U_q^{R}({\widehat{\mathfrak{sl}(m|n)}})$.
\end{lemm}
\begin{proof}
The algebra $U_q^{R}({\widehat{\mathfrak{gl}(m|n)}})$ is generated by the coefficients of the series $\{k_{i}^{\pm}(z),e_j^{\pm}(z), f_{j}^{\pm}(z)\mid1\leqslant i\leqslant m+n; 1\leqslant j\leqslant m+n-1\}$, so the coefficients of the series $k_{1}^{\pm}(z)$ together with those in \eqref{SL} also generate the algebra $U_q^{R}({\widehat{\mathfrak{gl}(m|n)}})$. Note that, for any $g^{\pm}$, the automorphism $\mu_{g^{\pm}}$ fixes all generators in (\ref{SL}) and $\mu_{g^{\pm}}(k_{1}^{\pm}(z))=g^{\pm}(z)k_{1}^{\pm}(z)$. 

By Proposition \ref{base} and \eqref{Gauss ii}-\eqref{Gauss ji}, any element $P\in U_q^{R}({\widehat{\mathfrak{gl}(m|n)}})$ is a polynomial in $k_{1,0}^{\pm}$, $k_{1,\pm 1}^{\pm}$, $k_{1,\pm 2}^{\pm}, \cdots$ and the other generators fixed by all $\mu_{g^{\pm}}$. We may assume the monomials in $P$ are ordered with $k_{i,\pm r}^{\pm}$ preceding ${e_{i}^{\pm}}^{(r)}$, which precede ${f_{i}^{\pm}}^{(r)}$. Suppose $P\in U_q^{R}({\widehat{\mathfrak{sl}(m|n)}})$, and $M$ be the maximal index $r$ such that $k_{1,\pm r}^{\pm}$ appears in $P$, and $K$ the highest power of any such $k_{1,\pm r}^{\pm}$. One can write:
\begin{align*}
P=\sum_{a}(k_{1,0}^{+})^{a_0^+}(k_{1,0}^{-})^{a_0^-}(k_{1,1}^{+})^{a_1^+}(k_{1,-1}^{-})^{a_1^-}\cdots
(k_{1,M}^{+})^{a_M^+}(k_{1,-M}^{-})^{a_M^-}K_aE_aF_a~,
\end{align*}
where $E_a$, $K_a$, $F_a$ are monomials in the generators fixed by $\mu_{g^{\pm}}$, and the sum is over all $2M+2$-tuples $a=(a_0^+,a_0^-,a_1^{+},a_1^{-},\cdots,a_M^{+},a_M^-)$ with $0\leqslant a_{i}^{\pm}\leqslant K$. Fix $g^{\pm}=1+\lambda z^{\mp M}$ with $\lambda\in\mathbb{C}^{*}$. Then we have
\begin{align*}
\mu_{g^{\pm}}(P)=\sum_{a}(k_{1,0}^{+})^{a_0^+}
(k_{1,0}^{-})^{a_0^-}(k_{1,1}^{+})^{a_1^+}(k_{1,-1}^{-})^{a_1^-}\cdots(\lambda k_{1,0}^{+}+k_{1,M}^{+})^{a_M^+}(\lambda k_{1,0}^{-}+k_{1,-M}^{-})^{a_M^-}K_aE_aF_a=P~.
\end{align*}
On the one hand, each monomial $K_aE_aF_a$ can be written as the linear combination of the monomials in \eqref{base:Uq}. On the other hand, by \eqref{Gauss ii}, $k_{1,r}^{\pm}={l_{11}^{\pm}}^{(r)}$. Due to the linear independence of the different monomials in \eqref{base:Uq} and the arbitrariness of $\lambda$, it follows that $k_{1,M}^{\pm}$ cannot appear in $P$. This implies the claim.
\end{proof}

We proceed to prove the isomorphism between the $R$-matrix and Drinfeld presentations for quantum affine superalgebra $U_q(\widehat{\mathfrak{sl}(m|n)})$.
\begin{prop}
The subalgebra $U_q^{R}({\widehat{\mathfrak{sl}(m|n)}})$ is isomorphic to the superalgebra $U_q^{D}({\widehat{\mathfrak{sl}(m|n)}})$.
\end{prop}
\begin{proof}
By restricting the isomorphism $\varphi$ to the generators of $U_q^{D}({\widehat{\mathfrak{sl}(m|n)}})$, we have
$$
\begin{aligned}
\varphi\left(\dot{K}_{i}^{\pm}(z)\right)&={k_{i}^{\pm}(zq^{\nu_i})}^{-1}k_{i+1}^{\pm}(zq^{\nu_i}),\\
\varphi\left(\dot{X}_{i}^{+}(z)\right)&=(q-q^{-1})^{-1}(e_i^{+}(z_{-}q^{\nu_i})-e_{i}^{-}(z_{+}q^{\nu_i})),\\
\varphi\left(\dot{X}_{i}^{-}(z)\right)&=(-1)^{|\alpha_i|}(q^{-1}-q)^{-1}(f_{i}^{+}(z_{+}q^{\nu_i})-
f_{i}^{-}(z_{-}q^{\nu_i})),
\end{aligned}
$$
for $1\leq i\leq m+n-1$. By Lemma \ref{SL generators}, 
the coefficients of these images generate the algebra  $U_q^{R}({\widehat{\mathfrak{sl}(m|n)}})$.
It follows that $U_q^{R}({\widehat{\mathfrak{sl}(m|n)}})\cong U_q^{D}({\widehat{\mathfrak{sl}(m|n)}})$.
\end{proof}

In the case $m=n$, the coefficients $b_i^{\pm} (i\geqslant 0)$ of the quantum Berezinian lie in the center of $U_q^{R}({\widehat{\mathfrak{sl}(n|n)}})$. We now introduce the quantum affine superalgebra associated with the classical Lie superalgebra $\boldsymbol{A}(n-1,n-1)$ as the quotient algebra:
$$
U_q^{R}({\widehat{\mathfrak{psl}(n|n)}}):=U_q^{R}({\widehat{\mathfrak{sl}(n|n)}})/
\langle B_q(L^{\pm}(z))=1\rangle=U_q^{R}({\widehat{\mathfrak{sl}(n|n)}})/\mathfrak{C}_{n|n}.
$$

\appendix

\section{Appendix: Commutation relations between Gaussian generators}
\subsection{Case $m=1,n=1$}\label{appe:mn11}

From \eqref{RLL2.1}--\eqref{RLL2.7}, the following relations can be obtained:
\begin{align}\label{RLL2.8}
 k_i^{ \pm}(z) k_i^{ \pm}(w)&=k_i^{ \pm}(w) k_i^{ \pm}(z), \quad i=1,2, \tag{A.1}
\\ \label{RLL2.9}
k_1^{+}(z) k_1^{-}(w)&=k_1^{-}(w) k_1^{+}(z), \tag{A.2}
\\ \label{RLL2.10}
\frac{w_{-} q- z_{+}q^{-1}}{z_{+} q-w_{-} q^{-1}} k_2^{+}(z) k_2^{-}(w)
&=\frac{w_{+} q- z_{-}q^{-1}}{z_{-} q-w_{+} q^{-1}} k_2^{-}(w) k_2^{+}(z),  \tag{A.3}
\\ \notag
k_1^{ \pm}(z) k_2^{ \pm}(w)&=k_2^{ \pm}(w) k_1^{ \pm}(z), 
\\ \notag
\frac{z_{ \pm}-w_{\mp}}{z_{ \pm} q-w_{\mp} q^{-1}} k_2^{\mp}(w)^{-1} k_1^{ \pm}(z)
&=k_1^{ \pm}(z) k_2^{\mp}(w)^{-1} \frac{z_{\mp}-w_{ \pm}}{z_{\mp} q-w_{ \pm}q^{-1}}.
\end{align}
Thus, all the relations between $k_{1}^{\pm}(z)$, $k_{2}^{\pm}(z)$ have been obtained.

From \eqref{RLL2.5}--\eqref{RLL2.7} and \eqref{RLL2.8}--\eqref{RLL2.10}, we can also get
\begin{align}\label{RLL2.14}
k_{1}^{\pm}(z)^{-1}f_{1}^{\pm}(w)k_{1}^{\pm}(z)&=\frac{zq-wq^{-1}}{z-w}f_{1}^{\pm}(w)+
\frac{w(q-q^{-1})}{w-z}f_{1}^{\pm}(z), \tag{A.5}
\\ \label{RLL2.15}
k_{1}^{\pm}(z)^{-1}f_{1}^{\mp}(w)k_{1}^{\pm}(z)&=\frac{z_{\mp}q-w_{\pm}q^{-1}}{z_{\mp}
-w_{\pm}}f_{1}^{\mp}(w)+\frac{w_{\pm}(q-q^{-1})}{w_{\pm}-z_{\mp}}f_{1}^{\pm}(z), \tag{A.6}
\\
k_{1}^{\pm}(z)e_{1}^{\pm}(w)k_{1}^{\pm}(z)^{-1}&=\frac{zq-wq^{-1}}{z-w}e_{1}^{\pm}(w)-
\frac{z(q-q^{-1})}{w-z}e_{1}^{\pm}(z), \tag{A.7}
\\ \label{RLL2.17}
k_{1}^{\pm}(z)e_{1}^{\mp}(w)k_{1}^{\pm}(z)^{-1}&=\frac{z_{\pm}q-w_{\mp}q^{-1}}{z_{\pm}
-w_{\mp}}e_{1}^{\mp}(w)-\frac{z_{\pm}(q-q^{-1})}{w_{\mp}-z_{\pm}}e_{1}^{\pm}(z). \tag{A.8}
\end{align}
Recall the definition of $X_{1}^{\pm}(z)$; with the help of \eqref{RLL2.14}--\eqref{RLL2.17}, we get
\begin{align}
k_{1}^{\pm}(z)^{-1}x_{1}^{-}(w)k_{1}^{\pm}(z)&=\frac{z_{\mp}q-wq^{-1}}{z_{\mp}-w}x_{1}^{-}(w), \tag{A.9}
\\
k_{1}^{\pm}(z)x_{1}^{+}(w)k_{1}^{\pm}(z)^{-1}&=\frac{z_{\pm}q-wq^{-1}}{z_{\pm}-w}x_{1}^{+}(w). \tag{A.10}
\end{align}

We can also find the following relations from \eqref{RLL2.1} and \eqref{RLL2.2}
\begin{equation}\tag{A.11}\label{RLL2.18}
k_1^{\pm}(z)f_1^{\pm}(z)k_1^{\pm}(w)f_1^{\pm}(w)=-\frac{zq^{-1}-wq}{zq-wq^{-1}}
k_1^{\pm}(w)f_1^{\pm}(w)k_1^{\pm}(z)f_1^{\pm}(z),
\end{equation}
\begin{equation}\tag{A.12} \label{RLL2.19}
k_1^{\pm}(z)k_1^{\pm}(w)f_1^{\pm}(w)=\frac{w(q-q^{-1})}{zq-wq^{-1}}k_1^{\pm}(w)k_1^{\pm}(z)
f_1^{\pm}(z)+\frac{z-w}{zq-wq^{-1}}k_1^{\pm}(w)f_1^{\pm}(w)k_1^{\pm}(z),
\end{equation}
\begin{equation}\tag{A.13}\label{RLL2.20}
k_1^{\mp}(z)f_1^{\mp}(z)k_1^{\pm}(w)f_1^{\pm}(w)=-\frac{z_{\pm}q^{-1}-w_{\mp}q}{z_{\pm}q-w_{\mp}q^{-1}}
k_1^{\pm}(w)f_1^{\pm}(w)k_1^{\mp}(z)f_1^{\mp}(z),
\end{equation}
\begin{equation}\tag{A.14}\label{RLL2.21}
k_1^{\mp}(z)k_1^{\pm}(w)f_1^{\pm}(w)=\frac{w_{\mp}(q-q^{-1})}{z_{\pm}q-w_{\mp}q^{-1}}
k_1^{\pm}(w)k_1^{\mp}(z)f_1^{\mp}(z)+\frac{z_{\pm}-w_{\mp}}{z_{\pm}q-w_{\mp}q^{-1}}
k_1^{\pm}(w)f_1^{\pm}(w)k_1^{\mp}(z).
\end{equation}
Using \eqref{RLL2.14} and \eqref{RLL2.15}, we have
\begin{align*}
\frac{zq^{-1}-wq}{z-w}f_1^{\pm}(z)f_1^{\pm}(w)-\frac{z(q-q^{-1})}{z-w}f_1^{\pm}(w)
f_1^{\pm}(w)&=\frac{wq-zq^{-1}}{z-w}f_1^{\pm}(w)f_1^{\pm}(z)\\
&+\frac{w(q-q^{-1})}{w-z}
\frac{wq-zq^{-1}}{zq-wq^{-1}}f_1^{\pm}(z)f_1^{\pm}(z),\\
\frac{z_{\mp}q^{-1}-w_{\pm}q}{z_{\mp}-w_{\pm}}f_1^{\pm}(z)f_1^{\mp}(w)-
\frac{z_{\mp}(q-q^{-1})}{z_{\mp}-w_{\pm}}f_1^{\mp}(w)f_1^{\mp}(w)&=
\frac{w_{\pm}q-z_{\mp}q^{-1}}{z_{\mp}-w_{\pm}}f_1^{\mp}(w)f_1^{\pm}(z)\\
&+\frac{w_{\pm}(q-q^{-1})}{w_{\pm}-z_{\mp}}\frac{w_{\pm}q-z_{\mp}q^{-1}}
{z_{\mp}q-w_{\pm}q^{-1}}f_1^{\pm}(z)f_1^{\pm}(z).
\end{align*}

Similarly, we can also get relations between $e_1^{\pm}(z)$ and $e_1^{\pm}(w)$:
\begin{align*}
\frac{zq^{-1}-wq}{z-w}e_1^{\pm}(w)e_1^{\pm}(z)+\frac{w(q-q^{-1})}{z-w}e_1^{\pm}(w)
e_1^{\pm}(w)&=\frac{wq-zq^{-1}}{z-w}e_1^{\pm}(z)e_1^{\pm}(w)\\
&-\frac{wq-zq^{-1}}{zq-wq^{-1}}\frac{z(q-q^{-1})}{w-z}e_1^{\pm}(z)e_1^{\pm}(z),\\
\frac{z_{\pm}q^{-1}-w_{\mp}q}{z_{\pm}-w_{\mp}}e_1^{\mp}(w)e_1^{\pm}(z)
+\frac{w_{\mp}(q-q^{-1})}{z_{\pm}-w_{\mp}}e_1^{\mp}(w)e_1^{\mp}(w)&=
\frac{w_{\mp}q-z_{\pm}q^{-1}}{z_{\mp}-w_{\mp}}e_1^{\pm}(z)e_1^{\mp}(w)\\
&-\frac{w_{\mp}q-z_{\pm}q^{-1}}{z_{\pm}q-w_{mp}q^{-1}}\frac{z_{\pm}(q-q^{-1})}
{w_{\mp}-z_{\pm}}e_1^{\pm}(z)e_1^{\pm}(z).
\end{align*}

Then, we can derive
\begin{align}\tag{A.15}
x_1^{\pm}(z)x_1^{\pm}(w)=-x_1^{\pm}(w)x_1^{\pm}(z).
\end{align}

From \eqref{RLL2.5}--\eqref{RLL2.7} and (\ref{RLL2.8}), we have
\begin{align*}
{k_2^{\pm}(w)}^{-1}f_1^{\pm}(z)k_2^{\pm}(w)=-\frac{z(q-q^{-1})}{z-w}f_1^{\pm}(w)
+\frac{zq^{-1}-wq}{z-w}f_1^{\pm}(z),\\
{k_2^{\mp}(w)}^{-1}f_1^{\pm}(z)k_2^{\mp}(w)=-\frac{z_{\mp}(q-q^{-1})}{z_{\mp}-w_{\pm}}
f_1^{\mp}(w)+\frac{z_{\mp}q^{-1}-w_{\pm}q}{z_{\mp}-w_{\pm}}f_1^{\pm}(z).
\end{align*}
Then, we obtain that
\begin{equation}\tag{A.16}
k_{2}^{\pm}(w)^{-1}x_{1}^{-}(z)k_{2}^{\pm}(w)=\frac{w_{\mp}q-zq^{-1}}{w_{\mp}-z}x_{1}^{-}(z).
\end{equation}

Same-type relations holds also for $e_{1}^{\pm}(z)$:
\begin{align*}
k_2^{\pm}(w)e_1^{\pm}(z){k_2^{\pm}(w)}^{-1}=\frac{w(q-q^{-1})}{z-w}e_1^{\pm}(w)
+\frac{zq^{-1}-wq}{z-w}e_1^{\pm}(z),\\
k_2^{\mp}(w)e_1^{\pm}(z){k_2^{\mp}(w)}^{-1}=\frac{w_{\mp}(q-q^{-1})}{z_{\pm}-w_{\mp}}
e_1^{\mp}(w)+\frac{z_{\pm}q^{-1}-w_{\mp}q}{z_{\pm}-w_{\mp}}e_1^{\pm}(z).
\end{align*}
Then, we have
\begin{equation}\tag{A.17}
k_{2}^{\pm}(w)x_{1}^{+}(z)k_{2}^{\pm}(w)^{-1}=\frac{w_{\pm}q-zq^{-1}}{w_{\pm}-z}x_{1}^{+}(z).
\end{equation}

From \eqref{RLL2.1} and \eqref{RLL2.2}, we get
\begin{align*}
&-\frac{z(q-q^{-1})}{zq-wq^{-1}}\left(k_2^{\pm}(z)+e_1^{\pm}(z)k_1^{\pm}(z)f_1^{\pm}(z)
\right)k_1^{\pm}(w)+\frac{z-w}{zq-wq^{-1}}k_1^{\pm}(z)f_1^{\pm}(z)e_1^{\pm}(w)k_1^{\pm}(w)\\
=&\frac{z-w}{zq-wq^{-1}}e_1^{\pm}(w)k_1^{\pm}(w)k_1^{\pm}(z)f_1^{\pm}(z)-
\frac{z(q-q^{-1})}{zq-wq^{-1}}\left(k_2^{\pm}(w)+e_1^{\pm}(w)k_1^{\pm}(w)f_1^{\pm}(w)
\right)k_1^{\pm}(z),\\
&-\frac{z_{\mp}(q-q^{-1})}{z_{\mp}q-w_{\pm}q^{-1}}\left(k_2^{\mp}(z)+e_1^{\mp}(z)
k_1^{\mp}(z)f_1^{\mp}(z)\right)k_1^{\pm}(w)+\frac{z_{\mp}-w_{\pm}}{z_{\mp}q-w_{\pm}q^{-1}}
k_1^{\mp}(z)f_1^{\mp}(z)e_1^{\pm}(w)k_1^{\pm}(w)\\
=&\frac{z_{\pm}-w_{\mp}}{z_{\pm}q-w_{\mp}q^{-1}}e_1^{\pm}(w)k_1^{\pm}(w)k_1^{\mp}(z)
f_1^{\mp}(z)-\frac{z_{\pm}(q-q^{-1})}{z_{\pm}q-w_{\mp}q^{-1}}\left(k_2^{\pm}(w)+
e_1^{\pm}(w)k_1^{\pm}(w)f_1^{\pm}(w)\right)k_1^{\mp}(z),\\
&k_1^{\pm}(z)f_1^{\pm}(z)k_1^{\pm}(w)=\frac{z-w}{zq-wq^{-1}}k_1^{\pm}(w)k_1^{\pm}(z)
f_1^{\pm}(z)-\frac{z(q-q^{-1})}{zq-wq^{-1}}k_1^{\pm}(w)f_1^{\pm}(w)k_1^{\pm}(z),\\
&k_1^{\mp}(z)f_1^{\mp}(z)k_1^{\pm}(w)=\frac{z_{\pm}-w_{\mp}}{z_{\pm}q-w_{\mp}q^{-1}}
k_1^{\pm}(w)k_1^{\mp}(z)f_1^{\mp}(z)-\frac{z_{\pm}(q-q^{-1})}{z_{\pm}q-w_{\mp}q^{-1}}
k_1^{\pm}(w)f_1^{\pm}(w)k_1^{\mp}(z).
\end{align*}
Then we can deduce that
\begin{align*}
\left[e_1^{\pm}(w),f_1^{\pm}(z)\right]=\frac{z(q-q^{-1})}{z-w}k_2^{\pm}(z){k_1^{\pm}(z)}^{-1}-\frac{z(q-q^{-1})}{z-w}
k_2^{\pm}(w){k_1^{\pm}(w)}^{-1},\\
\left[e_1^{\pm}(w),f_1^{\mp}(z)\right]=\frac{z_{\mp}(q-q^{-1})}{z_{\mp}-w_{\pm}}k_2^{\mp}(z){k_1^{\mp}(z)}^{-1}-
\frac{z_{\pm}(q-q^{-1})}{z_{\pm}-w_{\mp}}k_2^{\pm}(w){k_1^{\pm}(w)}^{-1}.
\end{align*}
Furthermore,
\begin{align*}
\left[e_1^{+}(w_-),f_1^{+}(z_+)\right]&=(q-q^{-1})\left(\sum_{i=0}^{\infty}(wz^{-1}q^{-c})^{i}
k_{2}^{+}(z_+){k_1^{+}(z_+)}^{-1}+\sum_{i=1}^{\infty}(zw^{-1}q^{c})^{i}
k_{2}^{+}(w_-){k_1^{+}(w_-)}^{-1}\right),\\
\left[e_1^{-}(w_+),f_1^{-}(z_-)\right]&=(q-q^{-1})\left(-\sum_{i=1}^{\infty}(zw^{-1}q^{-c})^{i}
k_{2}^{-}(z_-){k_1^{-}(z_-)}^{-1}-\sum_{i=0}^{\infty}(wz^{-1}q^{c})^{i}
k_{2}^{-}(w_+){k_1^{-}(w_+)}^{-1}\right),\\
\left[e_1^{+}(w_-),f_1^{-}(z_-)\right]&=(q-q^{-1})\left(-\sum_{i=1}^{\infty}(zw^{-1}q^{-c})^{i}
k_{2}^{-}(z_-){k_1^{-}(z_-)}^{-1}+\sum_{i=1}^{\infty}(zw^{-1}q^{c})^{i}
k_{2}^{+}(w_-){k_1^{+}(w_-)}^{-1}\right),\\
\left[e_1^{-}(w_+),f_1^{+}(z_+)\right]&=(q-q^{-1})\left(-\sum_{i=1}^{\infty}(zw^{-1}q^{c})^{i}
k_{2}^{+}(z_+){k_1^{+}(z_+)}^{-1}+\sum_{i=1}^{\infty}(zw^{-1}q^{-c})^{i}
k_{2}^{-}(w_+){k_1^{-}(w_+)}^{-1}\right).
\end{align*}
So,
\begin{equation}\tag{A.18}
\begin{aligned}
\left[x_1^{+}(w),x_1^{-}(z)\right]&=\left[e_1^{+}(w_-),f_1^{+}(z_+)\right]+\left[e_1^{-}
(w_+),f_1^{-}(z_-)\right]-\left[e_1^{+}(w_-),f_1^{-}(z_-)\right]-\left[e_1^{-}(w_+),
f_1^{+}(z_+)\right]\\
&=(q-q^{-1})\left(\delta\left(wz^{-1}q^{-c}\right)k_2^{+}(z_+){k_1^{+}(z_+)}^{-1}-
\delta\left(zw^{-1}q^{-c}\right)k_{2}^{-}(w_{+}){k_{1}^{-}(w_{+})}^{-1}\right).
\end{aligned}
\end{equation}

\subsection{Computation of \eqref{ser1 m=2} and \eqref{ser2 m=1}} \label{serre 12}

For the Case $m=1,n=2$, we have
\begin{align*}
(z-w)x_1^{+}(z) x_2^{+}(w)&=\left(z q^{-1}-w q\right) x_2^{+}(w) x_1^{+}(z),\\ 
\left(z q^{-1}-w q\right) x_1^{-}(z) x_2^{-}(w)&=(z-w) x_2^{-}(w) x_1^{-}(z), \\ 
(wq^{-1}-zq)x_{2}^{-}(z)x_{2}^{-}(w)&=(wq-zq^{-1})x_{2}^{-}(w)x_{2}^{-}(z),\\ 
(wq-zq^{-1})x_{2}^{+}(z)x_{2}^{+}(w)&=(wq^{-1}-zq)x_{2}^{+}(w)x_{2}^{+}(z).
\end{align*}
Then,
\begin{align*}
x_2^{+}(z_1)x_1^{+}(w)x_2^{+}(z_2)&=\frac{wq^{-1}-z_2q}{w-z_2}x_2^{+}(z_1)x_2^{+}(z_2)
x_1^{+}(w),\\
x_1^{+}(w)x_2^{+}(z_1)x_2^{+}(z_2)&=\frac{wq^{-1}-z_1q}{w-z_1}\frac{wq^{-1}-z_2q}{w-z_2}
x_2^{+}(z_1)x_2^{+}(z_2)x_1^{+}(w),\\
x_2^{+}(z_2)x_1^{+}(w)x_2^{+}(z_1)&=\frac{wq^{-1}-z_1q}{w-z_1}\frac{z_1q^{-1}-z_2q}
{z_1q-z_2q^{-1}}x_2^{+}(z_1)x_2^{+}(z_2)x_1^{+}(w),\\
x_2^{+}(z_2)x_2^{+}(z_1)x_1^{+}(w)&=\frac{z_1q^{-1}-z_2q}{z_1q-z_2q^{-1}}x_2^{+}(z_1)
x_2^{+}(z_2)x_1^{+}(w),\\
x_1^{+}(w)x_2^{+}(z_2)x_2^{+}(z_1)&=\frac{wq^{-1}-z_2q}{w-z_2}\frac{wq^{-1}-z_1q}{w-z_1}
\frac{z_1q^{-1}-z_2q}{z_1q-z_2q^{-1}}x_2^{+}(z_1)x_2^{+}(z_2)x_1^{+}(w),\\
x_2^{-}(z_1)x_1^{-}(w)x_2^{-}(z_2)&=\frac{w-z_2}{wq^{-1}-z_2q}x_2^{-}(z_1)x_2^{-}(z_2)
x_1^{-}(w),\\
x_1^{-}(w)x_2^{-}(z_1)x_2^{-}(z_2)&=\frac{w-z_1}{wq^{-1}-z_1q}\frac{w-z_2}{wq^{-1}-z_2q}
x_2^{-}(z_1)x_2^{-}(z_2)x_1^{-}(w),\\
x_2^{-}(z_2)x_1^{-}(w)x_2^{-}(z_1)&=\frac{w-z_1}{wq^{-1}-z_1q}\frac{z_1q-z_2q^{-1}}
{z_1q^{-1}-z_2q}x_2^{-}(z_1)x_2^{-}(z_2)x_1^{-}(w),\\
x_2^{-}(z_2)x_2^{-}(z_1)x_1^{+}(w)&=\frac{z_1q-z_2q^{-1}}{z_1q^{-1}-z_2q}x_2^{-}(z_1)
x_2^{-}(z_2)x_1^{-}(w),\\
x_1^{-}(w)x_2^{-}(z_2)x_2^{-}(z_1)&=\frac{z_1q-z_2q^{-1}}{z_1q^{-1}-z_2q}\frac{w-z_1}
{wq^{-1}-z_1q}\frac{w-z_2}{wq^{-1}-z_2q}x_2^{-}(z_1)x_2^{-}(z_2)x_1^{-}(w).
\end{align*}
In \eqref{ser2 m=1}, the coefficient of the term $x_2^{+}(z_1)x_2^{+}(z_2)x_1^{+}(w)$ becomes
\begin{align*}
&1-(q+q^{-1})\frac{wq^{-1}-z_2q}{w-z_2}+\frac{wq^{-1}-z_1q}{w-z_1}\frac{wq^{-1}-z_2q}{w-z_2}
+\frac{z_1q^{-1}-z_2q}{z_1q-z_2q^{-1}}\\
&-(q+q^{-1})\frac{wq^{-1}-z_1q}{w-z_1}\frac{z_1q^{-1}-z_2q}
{z_1q-z_2q^{-1}}+\frac{wq^{-1}-z_2q}{w-z_2}\frac{wq^{-1}-z_1q}{w-z_1}
\frac{z_1q^{-1}-z_2q}{z_1q-z_2q^{-1}}=0,
\end{align*}
and the coefficient of the term $x_2^{-}(z_1)x_2^{-}(z_2)x_1^{-}(w)$ becomes
\begin{align*}
&1-(q+q^{-1})\frac{w-z_2}{wq^{-1}-z_2q}+\frac{w-z_1}{wq^{-1}-z_1q}\frac{w-z_2}{wq^{-1}-z_2q}
+\frac{z_1q-z_2q^{-1}}{z_1q^{-1}-z_2q}\\
&-(q+q^{-1})\frac{w-z_1}{wq^{-1}-z_1q}\frac{z_1q-z_2q^{-1}}+\frac{z_1q-z_2q^{-1}}
{z_1q^{-1}-z_2q}\frac{w-z_1}{wq^{-1}-z_1q}\frac{w-z_2}{wq^{-1}-z_2q}=0.
\end{align*}
Therefore, \eqref{ser2 m=1} holds.
Similarly, we can verify  \eqref{ser1 m=2} holds.

\section*{Acknowledgments}
H. Lin is supported by the Postdoctoral Fellowship Program of CPSF (GZC20252014). H. Zhang is supported by the support of the National Natural Science Foundation of China (No. 12271332).

\end{document}